\documentclass[11pt]{amsart}

\usepackage{amscd,amssymb,amsthm,verbatim}
\usepackage{enumerate}
\usepackage{bm}
\usepackage[margin=1in]{geometry}
\usepackage{mathrsfs, graphicx} 
\usepackage{stmaryrd}
\usepackage{epigraph}

\usepackage[%
bookmarks=true,
colorlinks,
linkcolor=blue,
urlcolor=blue,
citecolor=blue,
plainpages=false,
pdfpagelabels,
final,
breaklinks=true\between
]{hyperref}
\usepackage{hyperref}
\usepackage[numbers,sort&compress]{natbib}

\numberwithin{equation}{section}

\theoremstyle{plain}
\newtheorem{thm}{Theorem}[section]
\newtheorem{prop}[thm]{Proposition}
\newtheorem{lem}[thm]{Lemma}
\newtheorem{cor}[thm]{Corollary}
\newtheorem{conj}[thm]{Conjecture}

\theoremstyle{definition}
\newtheorem{defn}[thm]{Definition}

\theoremstyle{remark}
\newtheorem{rk}[thm]{Remark}
\newcommand{\ve}{\varepsilon}

\newcommand{\Ric}{\mathrm{Ric}}

\newcommand{\dist}{\mathrm{dist}}

\newcommand{\loc}{\mathrm{loc}}
\newcommand{\spt}{\operatorname{spt}}
\newcommand{\Int}{\operatorname{int}}
\newcommand{\vol}{\operatorname{vol}}

\newcommand{\Div}{\operatorname{div}}

\newcommand{\csum}{\mathbin{\#}}

\DeclareFontFamily{U}{MnSymbolC}{}
\DeclareSymbolFont{MnSyC}{U}{MnSymbolC}{m}{n}
\DeclareFontShape{U}{MnSymbolC}{m}{n}{
	<-6>  MnSymbolC5
	<6-7>  MnSymbolC6
	<7-8>  MnSymbolC7
	<8-9>  MnSymbolC8
	<9-10> MnSymbolC9
	<10-12> MnSymbolC10
	<12->   MnSymbolC12}{}
\DeclareMathSymbol{\intprod}{\mathbin}{MnSyC}{'270}

\begin{document}
\title[Noncompact fill-ins of Bartnik data]
{Noncompact fill-ins of Bartnik data}
\author{Dan A.\ Lee, Martin Lesourd, and Ryan Unger}

\address{Graduate Center and Queens College, City University of New York, 365 Fifth Avenue, New York, NY 10016}
\email{dan.lee@qc.cuny.edu}

\address{Black Hole Initiative, Harvard University, Cambridge, MA 02138}
\email{mlesourd@fas.harvard.edu}

\address{Department of Mathematics, Princeton University, Princeton, NJ 08544}
\email{runger@math.princeton.edu}

\maketitle

\begin{abstract}
We generalize Y. Shi and L.-F.\ Tam's \cite{ShiTam} nonnegativity result for the Brown-York mass, by considering nonnegative scalar curvature (NNSC) fill-ins that need only be complete rather than compact. Moreover, the NNSC fill-ins need not even be complete as long the incompleteness is ``shielded'' by a region with positive scalar curvature and occurs occurs sufficiently far away.

We accomplish this by generalizing P.~Miao's~\cite{Miao02} positive mass theorem with corners to asymptotically flat manifolds that may have other complete ends, or possibly incomplete ends that are appropriately shielded. We can similarly extend other results on the compact NNSC fill-in problem to allow for complete (or shielded) NNSC fill-ins. In particular, we prove the following generalization of a theorem of Miao~\cite{Miao20}:
Given any metric $\gamma$ on a closed manifold $\Sigma^{n-1}$, there exists a constant $\lambda$ such that for any complete (or shielded) NNSC fill-in $(\Omega^n, g)$ of $(\Sigma^{n-1},\gamma)$, we have  $\min_\Sigma H \le \lambda$, where $H$ is the mean curvature of $\Sigma$ with respect to $g$.
\end{abstract}

\tableofcontents

\section{Introduction}

We first recall the positive mass theorem for manifolds with  arbitrary complete ends. 
\begin{thm}[Positive mass theorem with arbitrary complete ends]\label{pmt}
Let $(M^n,g)$ be a complete manifold with at least one asymptotically flat end $\mathcal E$, and suppose it has nonnegative scalar curvature. Assume that $M$ is spin or $3\le n\le 7$.
Then the ADM mass of $\mathcal E$ is nonnegative. Furthermore, if the mass is zero, then $(M,g)$ is isometric to Euclidean space. 
\end{thm}
This theorem was first proved in the spin case using a generalization of Witten's argument \cite{Witten} by R.~Bartnik and P.~Chru\'{s}ciel~\cite{BartnikChrusciel}. (In fact, their proof generalizes to initial data sets.) Without a spin assumption, the nonnegativity part of the theorem was first proved by the second two authors and S.-T.\@ Yau under an asymptotically Schwarzschild assumption~\cite{LUY21}, by generalizing the original approach of Schoen--Yau using solutions to the Plateau problem \cite{SY79PMT}. 

In a later paper, we relaxed the asymptotically Schwarzschild assumption to the standard asymptotically flat assumption~\cite[Theorem 1.5]{LLU22} by means of a density theorem. In \cite{LLU22}, we also gave a proof of the rigidity statement, and we considered various other situations, including incomplete manifolds, negative mass, and boundaries whose mean curvature has a ``bad'' sign. 

The density theorem was independently proved by Jintian Zhu \cite{Zhu22}. (See also~\cite{CLSZ21}.) 

An important feature of \cite{LUY21, LLU22} is that, in addition to allowing for arbitrary complete ends, the results allow for the possibility of \textit{incompleteness}, so long as this incompleteness is ``shielded'' from the asymptotically flat end by a region with sufficiently positive scalar curvature. 
\begin{defn}\label{def_shield}
Let $(M^n,g)$ be a Riemannian manifold, not assumed to be complete or to have nonnegative scalar curvature everywhere. We say that $(M, g)$ \emph{contains a scalar curvature shield $U_0 \supset U_1 \supset U_2$} if $U_0$, $U_1$, and $U_2$ are open subsets of $M$ such that 
$U_0 \supset \overline U_1$, $U_1\supset \overline U_2$,
 the closure of $U_0$ in $(M,g)$ is complete,
and we have the following:
\begin{enumerate}
    \item $R_g\ge0$ on $U_0$,
    \item $R_g \ge \kappa$ on $\overline U_1 \setminus U_2$, where $\kappa>0$,
    \item $H_{\partial U_0}\le \eta$, where $\eta>0$, where the mean curvature is measured with respect to the normal pointing toward $U_0$,\footnote{Our convention is such that the mean curvature of a standard sphere in Euclidean space ``measured with respect to the outward normal'' is positive.}
    \item $D_0 > \frac{4}{\kappa D_1}-\frac{2}{\eta}$, where $D_0=\dist_g(\partial U_0, U_1)$ and $D_1 =\dist_g(\partial U_1, U_2).$ \label{shield_condition}
\end{enumerate}
\end{defn}
The conditions stated in this definition subsume the conditions used in both Theorems 1.1 and 1.7 of~\cite{LLU22}. The following theorem is simply a re-statement of Theorems  1.1 and 1.7 of~\cite{LLU22} using this vocabulary. 
\begin{thm}\label{shielding} 
Let $(M^n,g)$ have an asymptotically flat end $\mathcal E$, and suppose it contains a scalar curvature shield $U_0\supset U_1\supset U_2 \supset \mathcal E$ with $\overline {U_0\setminus \mathcal E}$ compact. 
Assume that $U_0$ is spin or $3\le n\le 7$. 
Then the ADM mass of $\mathcal E$ is strictly positive. 
\end{thm}
The key point is that $(M,g)$ need not be complete nor have nonnegative scalar curvature to hold outside $U_0$.

\begin{rk}
Note that when $D_0 > \frac{4}{\kappa D_1}$, item~\eqref{shield_condition} of Definition~\ref{def_shield} holds no matter what $\eta$ is. This case corresponds to Theorem 1.1 of~\cite{LLU22}. We ignore the case when $\eta\le0$ in Definition~\ref{def_shield}, because in this case, positivity of ADM mass was already well-known to follow from the original Schoen--Yau  argument.
\end{rk}
The assumption that $U_0$ is spin or $3\le n\le 7$ appears because in~\cite{LLU22}, we are able to reduce the theorem to the positive mass theorem for initial data sets with boundary, which is known with these hypotheses~\cite{Herzlich, LLU21}. We also note the work of S.~Cecchini and R.~Zeidler, which proves 
 a theorem in the spin case that is similar in spirit~\cite{CecchiniZeidler21}, cf.~\cite[Corollary 1.6]{LLU22}.

The purpose of the current paper is to bring the philosophy underlying the two theorems stated above to the study of quasi-local mass and nonnegative scalar curvature fill-ins. Let $(\Omega^n,g)$ be a Riemannian manifold with scalar curvature $R_g$ and boundary $(\Sigma^{n-1},\gamma)$ with mean curvature $H$, defined with respect to the outgoing normal. 
(Throughout this article, we will typically not need to assume that $\Sigma$ is connected.)
The interplay between $\Omega$, $g$, $\gamma$, $R_g$, and $H$ has been the subject of a number of works. An important concept in this area is the \textit{quasi-local mass} of $(\Omega,g)$. Although there are a number of distinct definitions of this concept, varying in their properties and intended purposes, many of these are defined in terms of the \textit{Bartnik data} $(\Sigma,\gamma,H)$ associated to~$\Omega$. A well-known example is the Brown-York mass, which can be defined as follows. 
\begin{defn}
Let $(\Sigma^{n-1},\gamma)$ be a Riemannian manifold that is isometric to a hypersurface $\Sigma^{n-1}_0$ in Euclidean $\mathbb{R}^{n}$. Given a function $\eta$ on $\Sigma$, the \emph{Brown-York mass} of $(\Sigma^{n-1},\gamma, \eta)$ is given by 
\begin{equation}\label{BY}
    m_{\mathrm{BY}}(\Sigma, \gamma, \eta)=\frac{1}{(n-1)\omega_{n-1}}\int_{\Sigma} H_0-\eta,
\end{equation}
where $H_0$ is the mean curvature of $\Sigma_0$ in $\mathbb{R}^n$ with respect to the outward pointing normal. 
\end{defn}

\begin{rk}\label{Weyl}
The Weyl embedding theorem, proved by L.~Nirenberg~\cite{Nirenberg} and A.V.~Pogorelov~\cite{Pogorelov}, tells us that when $n=3$ and $(\Sigma, \gamma)$ is a 2-sphere with positive Gauss curvature, it can be isometrically embedded into Euclidean $\mathbb{R}^3$, and hence the Brown-York mass can be defined for such surfaces. For $n=3$, C.~Mantoulidis and P.~Miao \cite{MantoulidisMiao} defined a meaningful generalization of Brown-York mass to more general surfaces, using the concept of fill-ins by compact manifolds with nonnegative scalar curvature. Similar notions are studied in \cite{shi1, shi3}.  
\end{rk}

In a landmark paper \cite{ShiTam}, Yuguang Shi and Luen-Fai Tam proved a positivity and rigidity result for Brown-York mass, which may be regarded as a quasi-local version of the positive mass theorem.

\begin{thm}[Shi--Tam \cite{ShiTam}]\label{first}
Let $(\Omega^n,g)$ be a compact Riemannian manifold with boundary $(\Sigma^{n-1}, \gamma)$.  Assume that $R_g\geq 0$, $H_\Sigma>0$ (where $H_\Sigma$ is the mean curvature of $\Sigma$ with respect to the outward pointing normal), and that each component $\Sigma_i$ of $\Sigma$
isometrically embeds as a strictly convex hypersurface in $\mathbb{R}^{n}$. Then for each $i$,
\begin{equation}\label{firsteq}
   m_{\mathrm{BY}}(\Sigma_i, \gamma, H_\Sigma)\geq 0.
\end{equation}
Moreover, if any $m_{\mathrm{BY}}(\Sigma_i, \gamma, H_\Sigma)= 0$, then $(\Omega^n,g)$ is isometric to a subset of Euclidean space.
\end{thm}
Shi and Tam's original argument requires $\Omega$ to be spin, but the non-spin case is also true in all dimensions in which the positive mass is true. The nonnegativity follows from combining Shi and Tam's argument with~\cite{Miao02}, while the equality case was proved in~\cite{EichmairMiaoWang}, which also slightly relaxes the hypotheses of Theorem~\ref{first}.

Our first observation is that one does not need $\Omega$ to be compact in Theorem~\ref{first}. We can instead assume either completeness or the presence of a scalar curvature shield.

\begin{thm}\label{shieldingBY}
The nonnegativity of $m_\mathrm{BY}(\Sigma_i ,\gamma,H_\Sigma)$ in Theorem~\ref{first} holds, with the first sentence replaced by either of the following: 
\begin{itemize}
    \item Let $(\Omega^n, g)$ be a complete Riemannian manifold with compact boundary $(\Sigma^{n-1}, \gamma)$. 
    \item Let $(\Omega^n, g)$ be a Riemannian manifold with compact boundary $(\Sigma^{n-1}, \gamma)$, in the manifold sense, containing a scalar curvature shield $U_0\supset U_1\supset U_2\supset \Sigma$ with $\overline U_0$ compact.
\end{itemize}
Moreover, if any $m_{\mathrm{BY}}(\Sigma_i, \gamma, H_\Sigma)= 0$, then we must be in the first case, and $(\Omega^n,g)$ must be Ricci-flat. 
\end{thm}

\begin{rk}
For the equality case, if $n=3$, and each $\Sigma_i$ is a sphere, then an argument from~\cite[Section 5]{Miao02} implies that $(\Omega^n,g)$ must be isometric to a compact subset of Euclidean space. It would be desirable if we could conclude this more generally, but the known techniques for proving this in~\cite[Lemma 5]{EichmairMiaoWang} and~\cite{McFeron-Szekylihidi} do not extend to noncompact fill-ins in an obvious way.
\end{rk}
The proof of nonnegativity in Theorem~\ref{shieldingBY} is a straightforward combination of Shi and Tam's work~\cite{ShiTam} with generalizations of Theorems~\ref{pmt} and~\ref{shielding} that allow for the presence of ``corners.'' These generalization follows from approximating the metric with a corner by a $C^2$ metric with nonnegative scalar curvature (Theorem~\ref{approximation}). We provide these proofs in Section~\ref{pmtsection}, where we follow Pengzi Miao's argument from~\cite{Miao02} and combine it with techniques from~\cite{LLU22}. Some additional arguments are needed to handle the equality case.

Theorem~\ref{first} can be thought of as placing restrictions on possible ``fill-ins'' of Bartnik data. 
\begin{defn}
Given a compact Riemannian manifold $(\Sigma^{n-1}, \gamma) $ equipped with a function $\eta$,
  we say that a connected Riemannian manifold with boundary $(\Omega^n,g)$ is a \emph{nonnegative scalar curvature (NNSC) fill-in of $(\Sigma, \gamma, \eta)$} if $\partial \Omega$ can be diffeomorphically identified with $\Sigma$ in such a way that 
  $g$ induces the metric $\gamma$ on $\Sigma$, the induced mean curvature $H_{\Sigma}$ (with respect to the outward normal) is equal to $\eta$, and $R_g\ge0$.
\end{defn}

M.~Gromov conjectured that, ignoring the mean curvature, given any compact manifold $\Omega^n$ with boundary $\Sigma^{n-1}$ and any metric $\gamma$ on $\Sigma$, there exists a metric $g$ on $\Omega$ inducing $\gamma$ such that $R_g>0$. This was recently confirmed by Shi, Wenlong Wang, and Guodong Wei \cite{shi3}.  
Shortly thereafter, Miao used 
this result to prove the following (an important case of which already appeared in~\cite{shi3}).
\begin{thm}[Miao \cite{Miao20}]\label{miao2}
Let $\Omega^n$ be a compact manifold with boundary $\Sigma^{n-1}$. Given any Riemannian metric $\gamma$ on $\Sigma$, there exists a constant $\lambda$ depending only on $\gamma$ and (the topology of) $\Omega$ such that if
there exists a \emph{compact} NNSC fill-in of $(\Sigma,\gamma,\eta)$, then
\begin{equation}\label{fill in mean curvature}
    \min_\Sigma \eta \le \lambda.
\end{equation} 
\end{thm}
\begin{rk}
$\Sigma$ is allowed to be disconnected in Theorem \ref{miao2}.
\end{rk}
\begin{rk}
In Theorem \ref{miao2}, $n$ denotes a dimension in which the following statement holds: a \textit{compact} manifold of the form $T^n\csum X^n$ cannot admit positive scalar curvature. This statement is contained in \cite{SY79} for $3\leq n\leq 7$, and claimed for all $n$ in \cite{SY17}.
\end{rk}

Note that for the case where $(\Sigma, \gamma)$ satisfies the hypotheses of Theorem~\ref{first} (that is, each component isometrically embeds as a strictly convex hypersurface of Euclidean space), Theorem~\ref{miao2} is an immediate corollary. Also, Gromov previously proved that Theorem~\ref{miao2} holds if one only considers compact fill-ins which are spin~\cite{GromovFour}.

We observe that Theorem \ref{miao2} can be strengthened to apply to \emph{complete} fill-ins rather than merely compact ones. We prove it by combining Miao's argument in \cite{Miao20} with Otis Chodosh and Chao Li's theorem on nonexistence of complete positive scalar curvature metrics on $T^n \# X^n$~\cite{ChodoshLi} (and linking them together via Theorem~\ref{approximation}).  We can also prove it for \emph{shielded} fill-ins by proving a shielded version of Chodosh and Li's theorem (Theorem~\ref{ChodoshLiShield}).
\begin{thm}\label{miaocomplete}
Let $3\leq n\leq 7$. Given any Riemannian metric $\gamma$ on a closed manifold $\Sigma^{n-1}$,  
there exists a constant $\lambda$ depending only on $\gamma$ such that for any function $\eta$ on $\Sigma$, if there exists a NNSC fill-in of $(\Omega^n, g)$ of $(\Sigma,\gamma,\eta)$ satisfying either
\begin{itemize}
    \item $(\Omega, g)$ is \emph{complete}; or
    \item $(\Omega, g)$ contains a scalar curvature shield $U_0\supset U_1\supset U_2\supset \Sigma$ with $\overline U_0$ compact;
\end{itemize}
then
\begin{equation}
    \min_\Sigma \eta\le \lambda.
\end{equation} 
\end{thm}

One advantage of this theorem is that unlike Theorem \ref{miao2},
it does not require $\Sigma$ to be null-cobordant. In fact, no topological assumption at all is made. 
Conceptually, it ought to be easier to construct complete NNSC fill-ins than compact ones. Even for null-cobordant manifolds, there is usually no canonical way to fill it in topologically. (For example, given a torus, one must decide which homology classes will become trivial in the fill-in.) In contrast, the product $[0,\infty)\times\Sigma$ is a canonical candidate to carry a complete NNSC metric. Indeed, for simple examples, one can easily accomplish this with a warped product construction.\\ \\
\textbf{Question.} Do there exist examples of Bartnik data $(\Sigma^{n-1},\gamma,\eta)$, with null-cobordant $\Sigma$, admitting complete NNSC fill-ins but \emph{not} compact NNSC fill-ins?\\ 

We have defined a ``fill-in'' $\Omega$ so that its boundary is identified with $\Sigma$, but for most purposes, it is just as useful to have $\Omega$ whose boundary is the union of $\Sigma$ with other minimal components. (The results described in this article apply to such fill-ins, but we leave out the precise statements for simplicity.) For this reason, a better (but harder) version of the question above would be: Do there exist examples of Bartnik data admitting complete NNSC fill-ins but no compact NNSC fill-ins in this more general sense that allows minimal boundary components?




Gromov \cite[p.\ 163]{GromovFour} conjectured the following strengthening of Theorem~\ref{miao2}.
\begin{conj}[Gromov]\label{gromovconj}
Given any Riemannian metric $\gamma$ on a closed manifold $\Sigma^{n-1}$,  
there exists a constant $\Lambda$ depending only on $\gamma$ such that for any function $\eta$ on $\Sigma^{n-1}$, if
there exists a compact NNSC fill-in of $(\Sigma^{n-1}, \gamma, \eta)$, then 
\begin{equation}\label{gromovineq}
\int_\Sigma \eta \le  \Lambda.
\end{equation}
\end{conj}
It is natural, in view of the current paper, to consider the conjecture with ``compact" replaced by ``complete or shielded." \\ \indent 
Lastly, we note the recent paper of Jianchun Chu, Man-Chun Lee, and Jintian Zhu \cite{CLZ22} which proves a version of Theorem \ref{pmt} that allows for a singular set \textit{away} from the asymptotically flat end. It is likely that the ideas in \cite{CLZ22} can be combined with those here to prove statements about fill-ins that are both singular and noncompact. \\ 

\textbf{Acknowledgements.} D.~Lee thanks Harvard's Center of Mathematical Sciences and Applications for their hospitality while this research was being carried out. M.~Lesourd thanks the Gordon and Betty Moore Foundation, and the John Templeton Foundation for support. R.~Unger thanks the University of Cambridge for hospitality as this paper was being finished.

\section{Preliminaries}\label{prelim}
In this section, we recall some definitions and results from \cite[Section 2]{LLU22}.
\subsection{Definitions}

\begin{defn}\label{strinfty}
Let $M^n$ be a noncompact manifold (possibly with boundary). A \emph{topological structure of infinity} for $M$ is a choice of open subset $\mathcal{E}\subset M$ that does not intersect $\partial M$ and a diffeomorphism 
\begin{equation}\Phi:\mathcal E\to \Bbb R^n\setminus \overline{B}_{r_0}\end{equation} for some positive number $r_0$.
The coordinates $x^i$ define a flat metric on $\mathcal E$, which we extend arbitrarily to a smooth background metric on all of $M$ and denote by $\overline g$. 

\end{defn}

\begin{defn}\label{sobolev}
Let $(\mathcal E, \Phi)$ be a topological structure of infinity for $M^n$, and select $\overline g$ as in Definition~\ref{strinfty}.
Let $N$ be a closed subset of $M$ which contains $\mathcal E$ and such that $N\setminus \mathcal E$ is compact. (For example, we might take $N=\overline{\mathcal E}$.) Given $k\in\Bbb N$, $p\ge 1$, and $s\in\Bbb R$, we define the \emph{weighted Sobolev space} $W^{k,p}_s(N)$ to be the space of functions $u\in W^{k,p}_\loc(N)$ with finite norm 
\[\|u\|_{W^{k,p}_s(N)}=\|u\|_{W^{k,p}(N\setminus \mathcal E)}+\sum_{i=0}^k \|\partial^i u\|_{L^p_{s-i}(\mathcal E)},\]
where the weighted $L^p$ norm is defined by 
\[\|u\|_{L^p_s(\mathcal E)}=\left(\int_\mathcal E ||x|^{-s}u|^p \,\frac{dx}{|x|^n}\right)^\frac 1p,\]
and all integrals and norms are computed using the metric $\overline g$.\footnote{For fixed $N$, different choices of $\overline g$ will yield equivalent norms.}
Note that $|x|^s\notin L^p_s(\mathcal E)$ but $|x|^{s-\delta}\in L^p_s(\mathcal E)$ for any $\delta>0$. 
\end{defn}

\begin{defn}\label{defAF}
Let $(M^n,g)$ be a noncompact smooth Riemannian manifold, and let
$(\mathcal E, \Phi)$ be a topological structure of infinity for $M$. 
Let $p>n$ and $q>\frac{n-2}{2}$. We say that $(M, g)$ is \emph{$(p,q)$ asymptotically flat} if in the coordinates defined by $\Phi$, 
\[g_{ij}-\delta_{ij}\in W^{2,p}_{-q}(\mathcal E).\]
Furthermore, we assume that the scalar curvature of $g$, $R_g$, is integrable over the end $\mathcal E$.

Under these hypotheses, the \emph{ADM mass of $g$ in the end $\mathcal E$} is well-defined by 
\begin{equation}m_\mathrm{ADM}(\mathcal E, g)=\lim_{r\to\infty}\frac{1}{2(n-1)\omega_{n-1}}\int_{|x|=r} (g_{ij,i}-g_{ii,j})\frac{x^j}{|x|}\,d\mu_{S_r, \bar{g}}.\end{equation}
We will often write this as $m_\mathrm{ADM}(g)$ when $\mathcal E$ is understood from context.
\end{defn}

On any asymptotically flat manifold $(M, g)$ as above, one can construct a smooth proper function $\rho:M\to (0,\infty)$ with the following properties:
\begin{itemize}
    \item $\rho=|x|$ on $\mathcal{E}$ while $\rho$ is bounded away from $\mathcal{E}$, and 
    \item each $M_\sigma=\{\rho\ge\sigma\}$ is a complete subset of $M$ and the $M_\sigma$'s 
exhaust $M$ in the sense that $M=\bigcup_{\sigma>0} M_\sigma$. 
\end{itemize}
Note that when $\sigma$ is a regular value of $\rho$, $\partial M_\sigma$ is smooth, and thus $M_\sigma$ is a complete smooth asymptotically flat Riemannian manifold with boundary.

The following useful lemma is essentially Propositions 2.12 and 2.15 of~\cite{LLU22}, extended to include a case of interest to us.
In the following we decompose a function $V=V^+ - V^-$, where $V^+$ and $V^-$ represent the positive and negative parts of $V$ respectively. 

\begin{lem}\label{maintool}
Let $(M^n,g)$ be an $(p,q)$ asymptotically flat manifold, where $p>n$, $q> \frac{n-2}{2}$, and let $0<\sigma_0<\sigma_1$.  Let $V$ be a smooth integrable function on $M$ that is compactly supported in $M_{\sigma_0}$.  There exist constants $\ve_0>0$ and C, depending only on $n, p, q$,
 the $C^{1,\alpha}_{-q}(M_{\sigma_1})$ norm of $g-\overline g$, the $C^0(M_{\sigma_0/2}\setminus M_{\sigma_1})$ norm of $g$, and an upper bound on $\|V\|_{L^\infty(M_{\sigma_0} \setminus M_{3\sigma_1})}$, such that the following is true:

If we assume \textbf{either} that
\[ \|V^-\|_{L^\frac n2(M_{\sigma_0})}+\|V\|_{L^p_{-q-2}(M_{\sigma_0})}+\|V\|_{L^\frac{2n}{n+2}(M_{\sigma_0})} <\ve_0,\]
\textbf{or} that $V\ge0$ globally, 
then there exists a globally defined function $u$ on $M$ such that 
\[
 -a\Delta_g u + Vu  =0,     
\]
everywhere, where $a=4\frac{n-1}{n-2}$, such that $u-1\in W^{2,p}_{-q}(M_{\sigma_0})$. Moreover, $u$ has positive upper and lower bounds,
 and we can define the metric $\tilde g= u^\frac{4}{n-2}g$. This metric  $\tilde g$ is $(p,q)$ asymptotically flat, with scalar curvature
\begin{equation}
    R_{\tilde g}=(R_g-V)u^{-\frac{4}{n-2}}\label{conformalformula}
\end{equation}
and ADM mass
\begin{align}
    m_\mathrm{ADM}(\tilde g)&=    m_\mathrm{ADM}(g)-\frac{1}{2(n-1)\omega_{n-1}}\int_M Vu\,d\mu_g\label{massformula} \\
    \label{miaochangeinmass}
&\le m_{\textnormal{ADM}}(g) - \frac{1}{2(n-1)\omega_{n-1}} \int_{M_\sigma} \left(a|\nabla_g u|^2+Vu^2 \right) d\mu_g. 
\end{align}
Finally, the function $u$ satisfies
\begin{equation}\label{u-estimate}
\sup_M |u-1|+\|u-1\|_{W^{2,p}_{-q}(M_{2\sigma_1})}\le C\left(  \|V\|_{L^p_{-q-2}(M_{\sigma_0})} +\|V\|_{L^\frac{2n}{n+2}(M_{\sigma_0})}\right).
\end{equation}




\end{lem}
\begin{rk}
Since $u$ has positive upper and lower bounds, any end of $M$ that is complete with respect to $g$ will also be complete with respect to $\tilde{g}$.
\end{rk}

\begin{rk}
Being able to assume only $C^0$ control (or at most, $C^1$ control) over the metric in the compact set $M_{\sigma_0}\setminus M_{\sigma_1}$ (rather than $C^{1,\alpha}$) is important later on when we apply Lemma~\ref{maintool} to the metrics produced by Miao smoothing (see Lemma \ref{miaosmoothening}, item \eqref{C1-condition}).
\end{rk}
\begin{proof}
Consider the first case where $\|V^-\|_{L^\frac n2(M_{\sigma_0})}+\|V\|_{L^p_{-q-2}(M_{\sigma_0})}+\|V\|_{L^\frac{2n}{n+2}(M_{\sigma_0})} <\ve_0$. With the exception of the mass inequality~\eqref{miaochangeinmass}, which we will discuss later, the conclusion of Lemma~\ref{maintool} is essentially the one given in Proposition 2.16 of~\cite{LLU22}, with the estimate following directly from the one stated in Proposition 2.13 of~\cite{LLU22}. The only difference is that we are now claiming that the constant depends only on a $C^0$ bound (rather than a $C^{1,\alpha}$ bound) on the metric over the set $M_{\sigma_0/2}\setminus M_{\sigma_1}$. The claim essentially follows from using De Giorgi--Nash--Moser estimates instead of elliptic interior estimates over that region. 

In more detail, the proofs of Propositions 2.13 and 2.16 of~\cite{LLU22} clearly establish that 
\[ \sup_{M_{2\sigma_1}} |u-1|+\|u-1\|_{W^{2,p}_{-q}(M_{2\sigma_1})}\le C\left(  \|V\|_{L^p_{-q-2}(M_{\sigma_1})} +\|V\|_{L^\frac{2n}{n+2}(M_{\sigma_1})}\right),\]
where $C$ depends on the $C^{1,\alpha}_{-q}(M_{\sigma_1})$ norm of $g-\overline g$. So our goal is to bound $\sup_{M_{\sigma_0}\setminus M_{2\sigma_1}} |u-1|$, and then this will bound $\sup_M |u-1|$ via the maximum principle, since $V$ is supported in $M_{\sigma_0}$. The De Giorgi--Nash--Moser estimates \cite[Theorem 8.17]{GT} applied to the equation
\[-a\Delta_g(u-1)+V(u-1)=-V,\]
tell us that 
\begin{equation} \label{DGNM-estimate}
\sup_{M_{\sigma_0}\setminus M_{2\sigma_1}} |u-1| 
\le C \left( \|u-1\|_{L^\frac{2n}{n-2}(M_{\sigma_0/2}\setminus M_{3\sigma_1})} + \|V\|_{L^p(M_{\sigma_0/2}\setminus M_{3\sigma_1})}\right),
\end{equation}
where this $C$ depends only on the $C^0(M_{\sigma_0/2}\setminus M_{3\sigma_1})$ norm of $g$ and an upper bound on $\|V\|_{L^{\infty} (M_{\sigma_0/2}\setminus M_{3\sigma_1})}$.
 The $V$ term in~\eqref{DGNM-estimate} is obviously bounded by $\|V\|_{L^p_{-q-2}(M_{\sigma_0})}$, while we can bound $\|u-1\|_{L^\frac{2n}{n-2}(M_{\sigma_0/2}\setminus M_{3\sigma_1})}$
 by a constant times $\|V\|_{L^\frac{2n}{n+2}(M_{\sigma_0})}$ exactly as in equation (2.13) of \cite{LLU22}, noting that the constant only depends on $C^{0}(M_{\sigma_0/2})$ control of $g$ (and $\ve_0$) because the argument leading to (2.13) of \cite{LLU22} does not require the interior estimate (2.6) of \cite{LLU22}. The result follows.
 
So we need only discuss the case $V\ge0$. By the first part of Proposition 2.13 of~\cite{LLU22}, we know that for any $\sigma\in(0,\sigma_0)$ that is a regular value of $\rho$, there exists solution $u_\sigma$ to the problem 
\begin{align}
 \label{2a-new}   -a\Delta_g u_\sigma + Vu_\sigma  &=0 \quad\text{in }M_\sigma\\
 \label{2b-new}  \nu_g(u_\sigma) &=0 \quad\text{on }\partial M_\sigma,\\
u_\sigma-1 &\in W^{2,p}_{-q}(M_\sigma)
\end{align}
From the proof of~\cite[Proposition 2.13]{LLU22} together with the above De Giorgi--Nash--Moser argument, we can see that 
\[ \sup_{M_{\sigma}} |u_{\sigma}-1|+\|u_{\sigma}-1\|_{W^{2,p}_{-q}(M_{2\sigma_1})}\le C\left(
\|V\|_{L^p_{-q-2}(M_{\sigma_0})}\sup_{M_{\sigma}} |u_{\sigma}-1| + \|V\|_{L^p_{-q-2}(M_{\sigma_0})} +\|V\|_{L^\frac{2n}{n+2}(M_{\sigma_0})}\right).\]
Since $V\ge0$, the maximum principle implies that $0<u_\sigma \le1$, so we actually have the estimate  
\begin{equation}\label{V-positive}
\sup_{M_{\sigma}} |u_{\sigma}-1|+\|u_{\sigma}-1\|_{W^{2,p}_{-q}(M_{2\sigma_1})}\le 2C\left(
 \|V\|_{L^p_{-q-2}(M_{\sigma_0})} +\|V\|_{L^\frac{2n}{n+2}(M_{\sigma_0})}\right).
 \end{equation}
In particular, we have a uniform global supremum bound on $u_{\sigma}$ that is independent of $\sigma$, and each $u_\sigma$ satisfies the same elliptic equation~\eqref{2a-new}. Therefore there exists a sequence $\sigma_i\to0$
 such that $u_{\sigma_i}$ weakly converges in $W^{2,p}_{\mathrm{loc}}$ to a nonnegative globally defined function $u$ that satisfies the same equation~\eqref{2a-new}.

To see why $u$ has a positive lower bound, suppose that $\inf_M u = 0$. Then $\lim_{i\to\infty} \inf_M u_{\sigma_i} = 0$.  
The bound on $\|u_{\sigma_i} -1\|_{W^{2,p}_{-q}(M_{2\sigma_1})}$ from~\eqref{V-positive} implies that there is a fixed coordinate radius, independent of~$i$, in the asymptotically flat end, outside which $u_{\sigma_i}>\frac{1}{2}$. On the other hand, since $u_{\sigma_i}$ is harmonic outside $M_{\sigma_0}$ and satisfies a Neumann boundary condition, the maximum principle guarantees that for large $i$, the function $u_{\sigma_i}$ must be minimized somewhere in the compact set $K$ which is the part of $M_{\sigma_0}$ lying within the fixed coordinate radius described above.
It follows that there exists a point in $K$ where $u= 0$, but this contradicts the maximum principle for $u$ since $V\ge0$.

Now consider the metric $\tilde g= u^{\frac{4}{n-2}}g$. The conclusion that $\tilde g$ is  $(p,q)$ asymptotically flat and has mass given by~\eqref{massformula} follows exactly as in~\cite[Proposition 2.16]{LLU22}. 

The only thing left to check is~\eqref{miaochangeinmass}, whose proof is the same in both cases. Writing $ g_i= u_{\sigma_i}^{\frac{4}{n-2}}g$, a standard computation shows
\begin{align*}
     m_\mathrm{ADM}(g_i)
    -  m_\mathrm{ADM}(g)
    &=\lim_{r\to\infty}\frac{-2}{(n-2)\omega_{n-1}} \int_{|x|=r} u_{\sigma_i} \nu_g(u_{\sigma_i})\, d\mu_{S_r, g} \\
   &=\frac{-2}{(n-2)\omega_{n-1}} \int_{M} \Div_g (u_{\sigma_i} \nabla u_{\sigma_i})
  \, d\mu_{g} \\
&=   \frac{-1}{2(n-1)\omega_{n-1}} \int_{M_\sigma} \left(a|\nabla_g u_{\sigma_i}|^2+Vu_{\sigma_i}^2 \right) d\mu_g,
\end{align*}
but when we take the limit we only obtain an inequality because we have to apply Fatou's lemma.
\end{proof}

\section{Positive mass theorem with corners and arbitrary ends}\label{pmtsection}

We recall the definition of a manifold containing a ``corner'' across a hypersurface $\Sigma$ in the sense of \cite{Miao02}. 
\begin{defn}\label{corner}
Let $M^n$ be a manifold, and let $\Sigma$ be a smooth closed hypersurface of $M$ that separates $M$ in the sense that there exist closed subsets $\Omega_+, \Omega_-\subset M$ such that $M=\Omega_+ \cup \Omega_-$ and $\Sigma=\partial \Omega_\pm$. We say that a metric $g$ on $M$ \emph{admits a corner along $\Sigma$} if $g=g_+$ on $\Omega_+$ and $g=g_-$ on $\Omega_-$, such that $g_{\pm} \in C^{2,\alpha}_{\text{loc}} (\Int \Omega_\pm) \cap C^{2} (\Omega_\pm)$, and $g_+$ and $g_-$ induce the same metric on $\Sigma$. Let $H_{\pm}$ denote the mean curvature of $\Sigma$ computed with respect to $g_{\pm}$ with respect to the normal that points toward~$\Omega_+$.
\end{defn}

Although the following result is not stated explicitly in~\cite{Miao02}, it follows fairly directly from the arguments presented there.

\begin{thm}[Miao~\cite{Miao02}]\label{thm:Miao}
Let $(M^n, g)$ be a complete Riemannian manifold, all of whose ends are $(p,q)$ asymptotically flat, admitting a corner along a compact hypersurface $\Sigma$ as in Definition~\ref{corner}. Assume that $g$ has nonnegative scalar curvature away from $\Sigma$, and that $H_+ \le H_-$ along~$\Sigma$.

Then there exists an $\ve_0>0$ such that for any $0<\ve<\ve_0$, there exists a $C^2$
 metric $\tilde{g}$ on $M$ 
  which is also $(p,q)$ asymptotically flat in each end, has nonnegative scalar curvature everywhere, and satisfies the following hold:
 \begin{enumerate}
     \item $\| \tilde{g}-g\|_{C^0} < \ve$,
     \item $\| \tilde{g}-g\|_{W^{2,p}_{-q}(M\setminus U_{\ve})} < \ve$ where $U_{\ve}$ is the $\ve$-neighborhood of $\Sigma$,\label{W2p-close}
     \item $| m_\mathrm{ADM}(\mathcal E, \tilde g)-  m_\mathrm{ADM}( \mathcal E, g)|<\ve$ on each end $\mathcal E$.
 \end{enumerate}
 \end{thm}

Recall that the assumption $H_+ \le H_-$ is essentially a ``distributional NNSC'' assumption across the corner $\Sigma$. Clearly, as long as the positive mass theorem applies to $C^2$ asymptotically flat metrics on the manifold $M^n$, Theorem~\ref{thm:Miao} implies that the original metric $g$ has nonnegative mass, and this ``positive mass theorem with corners'' is the main result of~\cite{Miao02}. However, the approximation theorem above has other applications, and unlike the positive mass theorem with corners, it is not known to follow from a spinor argument when $M$ is spin. 

We generalize the above theorem to manifolds with arbitrary complete ends, or those containing a scalar curvature shield.
\begin{thm}[NNSC approximation for complete ends or shields]\label{approximation}
Let $(M^n, g)$ be a Riemannian manifold with an asymptotically flat end $\mathcal E$ and admitting a corner along $\Sigma$ as in Definition~\ref{corner}, and let $\sigma_0>0$ be small enough so that $\Sigma$ lies in the interior of $M_{\sigma_0}$ (as defined in Section~\ref{prelim}). Assume that $g$ has nonnegative scalar curvature on the complement of $\Sigma$, and that $H_+ \le H_-$ along~$\Sigma$.

Then there exists an $\ve_0>0$ such that for any $0<\ve<\ve_0$, there exists a metric $\tilde{g}$ that is $(p,q)$ asymptotically flat in $\mathcal E$, is $C^{2,\alpha}$ with respect to a new $C^{3,\alpha}$-differentiable structure on $M$, has nonnegative scalar curvature everywhere, and satisfies the following:
 \begin{enumerate}
     \item $\| \tilde{g}-g\|_{C^0} < \ve$,
     \item $\| \tilde{g}-g\|_{W^{2,p}_{-q}(M_{\sigma_0}\setminus T_{\ve})} < \ve$ where $T_{\ve}$ is the $\ve$-tubular neighborhood of $\Sigma$,
     \item $| m_\mathrm{ADM}(\mathcal E, \tilde g)-  m_\mathrm{ADM}( \mathcal E, g)|<\ve$.
 \end{enumerate}
 Moreover, if $(M, g)$ is complete, then $(M, \tilde g)$ is complete, and if $(M, g)$ contains a scalar curvature shield $U_0\supset U_1\supset U_2\supset \mathcal E$ with $\overline {U_0 \setminus \mathcal E}$ compact, $\Sigma \subset U_0$, and $\Sigma \cap (\overline U_1 \setminus U_2)=\emptyset$, then $\tilde g$ can be chosen so that $U_0\supset U_1\supset U_2$ is scalar curvature shield for $(M, \tilde g)$.
 \end{thm}
 \begin{rk}
Note that item~\eqref{W2p-close} implies that $|R_{\tilde g}-R_g|$ is small on $M\setminus U_{\ve}$.
\end{rk} 
 
Combining this with Theorem~\ref{pmt} and~\ref{shielding}, we obtain the following:

\begin{cor}[Positive mass theorem with corners and arbitrary ends]\label{pmtwithcorners}
Let $(M^n,g)$ and $\mathcal{E}$ satisfy the first paragraph of hypotheses of Theorem~\ref{approximation}, and assume that $3\le n\le 7$, or that $M$ is spin. If either
\begin{itemize}
    \item  $(M, g)$ is complete, or
    \item $(M, g)$ contains a scalar curvature shield $U_0\supset U_1 \supset U_2 \supset \mathcal E$ such that $\overline {U_0 \setminus \mathcal E}$ is compact,  $\Sigma \subset U_0$, and $\Sigma\cap (\overline U_1 \setminus U_2)=\emptyset$,
\end{itemize}
then $m_\mathrm{ADM}( \mathcal E, g)\ge0$.
\end{cor}

In order to treat the mass zero case, we have to work a bit harder:
\begin{thm}\label{corner-equality}
Let $(M^n,g)$ and $\mathcal{E}$ satisfy the hypotheses of Theorem~\ref{approximation} with corner $\Sigma$, and assume that $3\le n\le 7$, or that $M$ is spin. 
\begin{itemize}
    \item If $(M,g)$ is complete and $m_\mathrm{ADM}( \mathcal E, g)=0$, then it must be the case that  $H_+=H_-$ along $\Sigma$ and $(M,g)$ is Ricci-flat away from $\Sigma$. 
    \item If $(M, g)$ contains a scalar curvature shield $U_0\supset U_1 \supset U_2 \supset \mathcal E$ such that $\overline {U_0 \setminus \mathcal E}$ is compact,  $\Sigma \subset U_0$, and $\Sigma\cap (\overline U_1 \setminus U_2)=\emptyset$, then $m_\mathrm{ADM}( \mathcal E, g)>0$.
\end{itemize}
\end{thm}

The first step in the proof of Theorem~\ref{approximation} is to smooth out the metric along the corner $\Sigma$ in a small neighborhood of $\Sigma$. This process creates some negative scalar curvature, so the second step is to remove that negative scalar curvature via conformal change. The first step is the same as in~\cite[Section 3]{Miao02} while the second step is where we must alter the argument. The following is a slightly modified version of Proposition 3.1 of \cite{Miao02}.

\begin{lem}\label{miaosmoothening}
Let $(M^n, g)$ be a Riemannian manifold with an asymptotically flat end $\mathcal E$ and admitting a corner along $\Sigma$ as in Definition~\ref{corner}. Assume that $g$ has nonnegative scalar curvature away from $\Sigma$, and that $H_+ \le H_-$ along $\Sigma$. 
For some $\delta_0>0$, there exist $C^{3, \alpha}$ diffeomorphisms between a neighborhood $U_+$ of $\Sigma$ in $\Omega_+$ and $\Sigma\times [0,\delta_0)$ and between a neighborhood $U_-$ of $\Sigma$ in $\Omega_-$ and $\Sigma\times (-\delta_0, 0]$ that agree along $\Sigma$, such that if we define $U:= U_+\cup U_-$ and let $t$ denote the function in the second factor of the homeomorphism $\Phi: U\rightarrow\Sigma\times (-\delta_0,\delta_0)$ obtained by gluing together the diffeomorphisms above, then the ratio between $t$ and the signed distance function to $\Sigma$ is bounded between $1-
    \delta$ and $1+\delta$ on the complement of $\Sigma$, and we have the following:

 Let $\phi:\mathbb R\to [0,1]$ be a standard nonnegative bump function supported in $(-1,1)$ whose integral is~$1$. 
For any $\delta\in (0,\delta_0)$, there exists a metric $g_\delta$ on $M$ such that the following statements hold with a constant $C$ that is independent of $\delta$:
\begin{enumerate}
    \item $g_\delta$ is $C^{2,\alpha}$ on $M$ with respect to a new $C^{3,\alpha}$-differentiable structure,
    \item $g_{\delta}= g$ outside $U_\delta$, where $U_\delta:=\Phi^{-1}(\Sigma\times (-\delta, \delta))$, which is contained in the $2\delta$-neighborhood of $\Sigma$,
    \item $\displaystyle\lim_{\delta\to 0} \| g_\delta- g\|_{C^0(U_\delta)}=0$, 
    \item  $\| g_\delta\|_{C^1(U_\delta)}\le C$, \label{C1-condition}
    \item $| R_{g_{\delta}}| \le  C$ on $U_\delta \setminus U_{\frac{\delta^2}{100}}$
    \item For $p\in U_{\frac{\delta^2}{100}}$ with  $\Phi(p)=(x,t)$,
    \[\left| R_{g_{\delta}}(p) - (H_- - H_+)(x)\cdot  \left[\frac{100}{\delta^2}\phi\left(\frac{100t}{\delta^2}\right)\right]\right|\le C.\] 
\end{enumerate}
\end{lem}

 \begin{proof} For metrics that are $C^{4,\alpha}$ up to $\Sigma$ (which is stronger than the $C^{2,\alpha}$ assumption in Definition~\ref{corner}), the proof is exactly as in \cite{Miao02}.  
However, in the original proof given in~\cite{Miao02}, the function $t$ is taken to be exactly the signed distance function, which corresponds to working in Fermi coordinates (also called geodesic normal coordinates or Gaussian coordinates) with respect to the hypersurface $\Sigma$. Unfortunately, writing a metric in Fermi coordinates leads to a loss of differentiability of the metric, because the coordinates themselves are only $C^{1,\alpha}$. (Therefore, the metric cooefficients in Fermi coordinates might only be $C^{0,\alpha}$. See, for example, \cite{DTK} for discussion of this phenomenon.)
  Because of this, the proof in~\cite{Miao02} should technically require a $C^{4,\alpha}$ regularity assumption on $g$ (rather than the $C^{2,\alpha}$ assumption in Definition~\ref{corner}) in order to guarantee that $\tilde{g}$ is $C^2$. We do not believe this to be an essential problem, and it can probably be cured in various ways, but for the sake of being thorough, we will describe one way to solve the problem.

Since the arguments in~\cite{Miao02} prove the case when Fermi coordinates \emph{are} $C^{3,\alpha}$, we will simply reduce to this case via approximation. (For this part of the argument, we are only working on one side of $\Sigma$ at a time.)
Andersson and Chru\'{s}ciel have proved that one can always construct \emph{almost-Fermi} coordinates that are $C^{3,\alpha}$ \cite[Appendix B]{AnderssonChrusciel}. For the definition and properties of almost-Fermi coordinates, we refer the reader to \cite{AnderssonChrusciel} or \cite{LLU21}. In actual Fermi coordinates $x^1,\ldots, x^n$, we take $x^n$ to be the signed distance function to $\Sigma$ and we can arrange that $g_{nn}=1$ and $g_{in}=0$ \emph{exactly}. In almost-Fermi coordinates, taking $t=x^n$, we instead only have that $g_{nn}-1$ and $g_{in}$ are $O(t^{2+\alpha})$. Suppose we have written the metric $g$ in almost-Fermi coordinates. We define a new metric $\tilde{g}$ via
\begin{align*}
    \tilde{g}_{ij} &= g_{ij} \text{ for }i,j<n\\
    \tilde{g}_{in} &=\tilde{g}_{ni}=\chi g_{in}\text{ for }i<n \\
    \tilde{g}_{nn} &=1+ \chi (g_{nn}-1),
\end{align*}
where $\chi$ is a cutoff function that is equal to $1$ outside some small neighborhood of $\Sigma$ and $0$ inside some smaller neighborhood of $\Sigma$. The point of $\tilde{g}$ is that it \emph{does} have the property that $t=x^n$ \emph{is} the distance function for $\tilde{g}$, and thus the arguments of~\cite{Miao02} can be applied to it directly.
 
If the size of the neighborhood (and hence $\chi$ and $\tilde{g}$) is controlled by $\delta$, then $\chi$ can be chosen so that $\nabla\chi=O(\delta^{-1})$ and $\nabla^2 \chi=O(\delta^{-2})$. We claim that the mean curvature of $\Sigma$ with respect to $\tilde{g}$ is the same as it was with respect to $g$ and that the scalar curvature $R_{\tilde{g}}$ is bounded independently of~$\delta$. If both of these claims hold, one can see that deforming $g$ to $\tilde{g}$ and then applying the smoothing process from~\cite{Miao02} gives the desired result.

To see why the claim is true, first observe that the second order vanishing of $g_{nn}-1$ and $g_{in}$ implies that their $t$-derivatives vanish at $\Sigma$, and then the $C^2$ bound on $g$ implies that those derivatives are actually $O(t)$, or globally $O(\delta)$. We can now see that the mean curvature of $\Sigma$ is unchanged when we switch from $g$ to $\tilde{g}$ since both metrics have the exact same normal vectors and connection coefficients along $\Sigma$. To see why $R_{\tilde g}$ is bounded is slightly more complicated. Writing the scalar curvature in coordinates, the worst behaved term is when $t$-derivatives hit the $\chi$ terms. If two of them hit $\chi$, that yields a $O(\delta^{-2})$ contribution, but it must be multiplied by either $g_{in}$ or $g_{nn}-1$, both of which are $O(\delta^{2+\alpha})$. The worst case scenario is when one $t$-derivative hits $\chi$ and one $t$-derivative hits $g_{in}$ or $g_{nn}-1$, but based on what we discussed above, such a term will be $O(\delta^{-1})$ times $O(\delta)$, so this term is also bounded. 

Finally, we make a brief remark on the change of differentiable structure. The construction of $g_\delta$ in~\cite{Miao02} technically involves taking the pushforward of $g$ to the space  $\Sigma\times[ (-\delta_0,0)\cup(0, \delta_0)]$, which actually extends to a Lipschitz metric on $\Sigma\times(-\delta_0, \delta_0)$, and then mollifying it to obtain a smooth metric $h_\delta$ on $\Sigma\times(-\delta_0, \delta_0)$, but one cannot simply pullback $h_\delta$ to obtain a smooth metric on $U$ because $\Phi$ is only a homeomorphism and not a diffeomorphism. However, if we use $\Phi$ to \emph{define} a new smooth structure on $U$, then then that pullback, $g_\delta=\Phi^*(h_\delta)$, will (trivially) be smooth on $U$ with respect to the new smooth structure. We then define a new differentiable atlas for $M$ by taking the new coordinate charts for $U$ and combining them with the old coordinate charts for $M$ that do not intersect $\Sigma$. Since $\Phi$ is $C^{3,\alpha}$ away from $\Sigma$, that means that the transition maps for this atlas will be $C^{3,\alpha}$, and hence we have defined a new  $C^{3,\alpha}$ differentiable structure on $M$, with respect to which $g_\delta$ is $C^{2,\alpha}$.
 \end{proof}

\begin{proof}[Proof of Theorem~\ref{approximation}]

Let $(M, g)$, $\Sigma$, and $\sigma_0$ be as in the hypotheses of Theorem~\ref{approximation}, and construct the family $g_\delta$ described in Lemma~\ref{miaosmoothening}. The next step is to modify $(M,g_\delta)$ by a conformal deformation.
We seek a positive conformal factor $u_\delta$ asymptotic to $1$ in the asymptotically flat end such that 
\begin{equation} -a\Delta_{g_\delta} u_\delta -(R_{g_\delta})^- u_\delta = 0,\label{MiaoPDE1}\end{equation}
where $a=4\frac{n-1}{n-2}$ and $(R_{g_\delta})^-$ denotes the negative part of $R_{g_\delta}$. For sufficiently small $\delta>0$, the existence of $u_\delta$ is guaranteed by Lemma~\ref{maintool}. 
To see why, observe that if we choose $V=-(R_{g_\delta})^-$, then  Lemma~\ref{miaosmoothening} and the assumption  that $H_- - H_+ \ge0$ implies that $|V|<C$ on $U_\delta$ and $V$ vanishes outside $U_\delta$, and thus
\[ \|V^-\|_{L^\frac n2(M_{\sigma_0})}+\|V\|_{L^p_{-q-2}(M_{\sigma_0})}+\|V\|_{L^\frac{2n}{n+2}(M_{\sigma_0})}\to 0\text{ as }\delta\to 0.\]
So by Lemma \ref{maintool}, the desired $u_\delta$ exists, it  satisfies 
\[ \sup_{M}|u_\delta-1| + \|u_\delta-1\|_{W^{2,p}_{-q}(M_{\sigma_0})}\to 0\text{ as }\delta\to0,\]
and the conformal metric $\tilde{g}_\delta = u_\delta^{\frac{4}{n-2}} g_\delta$ has scalar curvature
\[ R_{\tilde{g}_\delta} = (R_{g_\delta})^+ u_\delta^{-\frac{4}{n-2}}\ge0,
\]
and ADM mass
\begin{equation}
    m_\mathrm{ADM}(\tilde g_\delta)=    m_\mathrm{ADM}( g)+\frac{1}{2(n-1)\omega_{n-1}}\int_M (R_{g_\delta})^- u_\delta\,d\mu_{g_\delta}.
\end{equation}
From these estimates and the constructions, we see that the three claims of Theorem~\ref{approximation} follow by taking $\delta$ sufficiently small.
Note that if $(M, g)$ is complete, so is $(M, \tilde{g}_\delta)$ since $\sup_M |u_\delta-1|$ is small.

In the case where $(M, g)$ has a scalar curvature shield $U_0\supset U_1\supset U_2\supset \mathcal E$ with $\Sigma \cap (\overline U_1 \setminus U_2)=\emptyset$, we can just choose $u_\delta$ to be the solution of~\eqref{MiaoPDE1} in $U_0$ with Neumann boundary condition at $\partial U_0$. (The three claims of Theorem~\ref{approximation} follow just as before, but without using the full force of Lemma~\ref{maintool}.) we can see that $U_0\supset U_1\supset U_2$ is still a scalar curvature shield for $(M, \tilde{g}_\delta)$, for sufficiently small $\delta$, as follows: First, the $C^0$ convergence of $\tilde{g}_\delta$ to $g$ means that the distances $\tilde{D}_0$ and $\tilde{D}_1$ for the metric $\tilde{g}_\delta$ (in Definition~\ref{def_shield}) can be made arbitrarily close to the $D_0$ and $D_1$ for the original metric $g$. Since $\Sigma \cap (\overline U_1 \setminus U_2)=\emptyset$ it follows that on $\overline U_1 \setminus U_2$, we have  
 \[ R_{\tilde{g}_\delta} =R_g u_\delta^{-\frac{4}{n-2}}\to R_g
\text{ uniformly as }\delta\to0. \]
Finally, since $\Sigma$ does not touch $\partial U$, the mean curvature of $\partial U_0$ with respect to $\tilde{g}_\delta$ is related to the original mean curvature with respect to $g$ via
\[ H_{\tilde{g}_\delta} = u_\delta^{\frac{-2}{n-2}}H_g \to H_g \text{ uniformly as }\delta\to0.
\]
Putting it all together, we see that we will still have the inequality in \eqref{shield_condition} for small enough $\delta$.
\end{proof}

\begin{rk}\label{actuallysmooth}
If the original metric $g$ is $C^\infty$ (or $C^{k,\alpha}$) up to the corner $\Sigma$ but not across it, then the  $\tilde{g}$ in the conclusion of Theorem~\ref{approximation} can also be chosen to be $C^\infty$ (or $C^{k,\alpha}$) on all of $M$, with respect to a new $C^\infty$-differentiable (or $C^{k+1,\alpha}$-differentiable) structure on $M$. There is one sticky point in this argument, which is that the potential function $-(R_{g_\delta})^-$ in equation \eqref{MiaoPDE1} is 
at best Lipschitz continuous, no matter how smooth $g_\delta$ is, because the 
 ``negative part" function $t^-=\max\{-t,0\}$ only a Lipschitz rather than smooth function of $t\in\Bbb R$. So even if $g$ starts out $C^\infty$, the metric $\tilde g$ constructed by Theorem~\ref{approximation} will not be smooth. However, this loss of regularity can be remedied as follows:
  Let $f:\Bbb R\to \Bbb R$ be a $C^\infty$ function with the following properties: $f(t)=t$ for $t\le \frac 12$, $f(t)=1$ for $t\ge 1$, and $f(t)\le \min\{t,1\}$ for every $t\in \Bbb R$. If we take $V=\chi_\delta f(R_{g_\delta})$, where $\chi_\delta$ is a cutoff supported in $U_{2\delta}$ and equal to $1$ on $U_{\delta}$, and continue with the rest of the proof of Theorem~\ref{approximation}, we will obtain a $\tilde{g}$ with the desired regularity.
\end{rk}

\begin{proof}[Proof of Theorem~\ref{corner-equality}]

We begin with the case where $(M, g)$ is complete. We will first prove that if $H_+(p)<H_-(p)$ for some point $p\in \Sigma$, then the mass must be strictly positive. This argument is the same one given in~\cite{Miao02}, which we include for the sake of completeness.

Let $\chi_\delta:M_\sigma\to [0,1]$ be a smooth cut-off function with $\chi_\delta =1$ in $\Sigma\times (-\delta,\delta)$ and $\chi_\delta=0$ outside  $\Sigma\times (-2\delta,2\delta)$.
Consider the metric $\tilde{g}_\delta$ constructed in the proof of Theorem~\ref{approximation}. The assumption that $H_+>H_-$ at some point $p\in\Sigma$ somewhere comes into play because it implies, via Lemma~\ref{miaosmoothening}, that ${R}_{\tilde{g}_\delta}$ is large and positive near $p$. 

We use a conformal transformation to make $\tilde{g}_\delta$ scalar-flat within $\Sigma\times (-\delta,\delta)$. This is done by solving  
\begin{equation}
\label{3a} - a\Delta_{\tilde{g}_{\delta}} v_{\delta} +\chi_\delta {R}_{\tilde{g}_\delta} v_{\delta} =0,
\end{equation}
for some function $v_\delta$ asymptotic to $1$ in the asymptotically flat end. Lemma~\ref{maintool} with $V= \chi_\delta {R}_{\tilde{g}_\delta}\ge0$ guarantees that $v_\delta$ exists and the new  $\hat{g}_\delta =v_\delta^{\frac{4}{n-2}}\tilde{g}_\delta$ is complete and asymptotically flat and has nonnegative scalar curvature. In particular, Theorem~\ref{pmt} imples that $ m_{\mathrm{ADM}}(\hat{g}_\delta)\ge0$,  and combining this with~\eqref{miaochangeinmass}, we obtain
\begin{equation}\label{energy-integral}
 m_{\mathrm{ADM}}(\tilde{g}_\delta) \ge  \frac{1}{2(n-1)\omega_{n-1}}\int_M \left( a|\nabla_{\tilde{g}_\delta} v_\delta|^2+ \chi_\delta {R}_{\tilde{g}_{\delta}} v^2_\delta \right)d\mu_{\tilde{g}_\delta}.
\end{equation} 

Since $m_{\mathrm{ADM}}(\mathcal{E},\tilde{g}_\delta)$ converges to $m_{\mathrm{ADM}}(\mathcal{E},g)$ as $\delta \to 0$ it remains to show that the right hand side of~\eqref{energy-integral} has a strict positive lower bound as $\delta\to 0$. If it does not then we obtain a contradiction in the exact same manner as in \cite[Proposition 4.2]{Miao02}. We omit the details, but the basic idea is the following: If there were no lower bound than there would be a sequence of $\delta$'s such that both  $\int_M |\nabla_{\tilde{g}_\delta} v_\delta|^2\, d\mu_{\tilde{g}_\delta}$ and $\int_M \chi_\delta {R}_{\tilde{g}_{\delta}} v^2_\delta \,d\mu_{\tilde{g}_\delta}$ approach zero. Without loss of generality, assume $\mathcal E\subset\Omega_+$. 
The energy of $v_\delta$ approaching zero would imply that $v_\delta$ uniformly converges to the constant function $1$ on compact subsets of the interior of $\Omega_+$. Although this does not give a positive lower bound on $v_\delta$ near $\Sigma$, it does imply a lower bound on a set large enough to contradict smallness of 
$\int_M \chi_\delta {R}_{\tilde{g}_{\delta}} v^2_\delta \,d\mu_{\tilde{g}_\delta}$, using our lower bound on  ${R}_{\tilde{g}_{\delta}}$. See~\cite[Proposition 4.2]{Miao02} for details.

Next we will prove that zero mass implies that $R_g$ vanishes away from the corner $\Sigma$. Our argument is different from the one used to prove the analogous fact in~\cite{Miao02}. For this we choose any cutoff function $\chi$ compactly supported away from $\Sigma$. We consider the same $g_\delta$ as in Lemma~\ref{miaosmoothening}, but we define $u_\delta$ and $\tilde g_\delta= u_\delta^{\frac{4}{n-2}} g_\delta$ differently from in the proof of Theorem~\ref{approximation}. Instead we choose $u_\delta$ to be a solution to
\[ -a\Delta_{g_\delta} u_\delta+ [\chi R_g -(R_{g_\delta})^-] u_\delta = 0\]
that is asymptotic to $1$ in $\mathcal E$.
Once again, Lemma~\ref{maintool} implies that $u_\delta$ exists, and that  $\tilde g_\delta$ is complete and asymptotically flat and has nonnegative scalar curvature, as long as $\chi\le1$ is selected to be sufficiently small (independently of $\delta$). So Theorem~\ref{pmt} combined with~\eqref{massformula} gives
\begin{align*}
0&\le m_{\mathrm{ADM}}(\tilde g_\delta) \\
&=m_{\mathrm{ADM}}(g_\delta) 
- \frac{1}{2(n-1)\omega_{n-1}} \int_M [ \chi R_g -(R_{g_\delta})^-]u_\delta \,d\mu_{g_\delta}\\
&=  \frac{-1}{2(n-1)\omega_{n-1}} \int_M [ \chi R_g -(R_{g_\delta})^-]u_\delta \,d\mu_{g_\delta},
\end{align*} 
since $g_\delta$ has the same mass as $g$, which we assumed has zero mass.
For small enough $\chi$ and $\delta$, we know that $\tfrac{1}{2}<u_\delta<\tfrac{3}{2}$, and therefore since we know that $\|(R_{g_\delta})^-\|_{L^1(M)}\to0$ as $\delta\to0$, it follows that 
\[ 0 \ge \int_M \chi R_g \,d\mu_{g_\delta}.\]
Hence $R_g$ vanishes on the support of $\chi$.

Finally, we prove that zero mass implies that $g$ is Ricci-flat away from $\Sigma$. (Note that~\cite{Miao02} does not contain this argument.) Again we choose a smooth cutoff $\chi$ supported away from $\Sigma$. Define $\bar g_t = g_{t^2} + t\chi \Ric_g$, where we allow $t$ to be negative but not $0$, and $g_{t^2}$ is just the $g_\delta$ constructed in Lemma~\ref{miaosmoothening} but with $\delta=t^2$. Next we define $u_t$ to be a solution to
\[ -a\Delta_{\bar g_t} u_t+ R_{\bar g_t} u_t = 0\]
that is asymptotic to $1$ in $\mathcal E$. Note that since we now have $H_+= H_-$ along $\Sigma$ and $R_g=0$ on the complement of $\Sigma$, for small enough $\delta$, $R_{\bar g_t}$ will have the smallness required for the existence of $u_t$ and completeness of $\tilde g_t := u_t^{\frac{4}{n-2}}\bar g_t$ in Lemma~\ref{maintool}, and $\tilde g_t$ will be asymptotically flat and scalar-flat.

Using Theorem~\ref{pmt} and~\eqref{massformula}, we have 
\begin{align}
0&\le m_{\mathrm{ADM}}(\tilde g_t)\nonumber \\
&=m_{\mathrm{ADM}}(\bar g_t) 
- \frac{1}{2(n-1)\omega_{n-1}} \int_M R_{\bar g_t} u_t \,d\mu_{\bar g_t}\nonumber \\
&=  \frac{-1}{2(n-1)\omega_{n-1}} \int_M R_{\bar g_t} u_t \,d\mu_{\bar g_t}, \label{integral_scalar_inequality}
\end{align} 
since $\bar g_t$ has the same mass as $g$, which we assumed has zero mass.

Consider the limit
\begin{equation}\label{limit_to_compute}
\lim_{t\to0} \int_M \frac{R_{\bar g_t} u_t}{t} d\mu_{\bar g_t}.
\end{equation}

Note that the scalar curvature $R_{\bar g_t}$ is made up of a contribution from the neighborhood $\Sigma\times(-t^2, t^2)$ and a contribution from the support of $\chi$. Since both $u_t$ and $R_{\bar g_t}$ are uniformly bounded and the neighborhood  $\Sigma\times(-t^2, t^2)$ has volume on the order of $t^2$, it does not contribute to the limit. Therefore
\[ \lim_{t\to0} \int_M \frac{R_{\bar g_t} u_t}{t} d\mu_{\bar g_t} = \lim_{t\to0} \int_{\mathrm{spt} \chi} \frac{R_{\bar g_t} u_t}{t} d\mu_{g} =\int_M \chi |\Ric_g|^2\,d\mu_g,
\]
where the second equality follows from that fact that $u_t\to1$ uniformly by~\eqref{u-estimate}, and a standard computation. (For example, see \cite[page 96]{Lee}. The Ricci curvature shows up because we are effectively linearizing the scalar curvature operator at $g_0$.) On the other hand, 
by inequality~\eqref{integral_scalar_inequality}, once we know that the limit~\eqref{limit_to_compute} exists, it must be zero, and hence $\Ric_g$ vanishes on the support of $\chi$, proving the result.

We now consider the second case, where $(M, g)$ contains a scalar curvature shield  $U_0\supset U_1 \supset U_2 \supset \mathcal E$ such that $\overline {U_0 \setminus \mathcal E}$ is compact,  $\Sigma \subset U_0$, and $\Sigma\cap (\overline U_1 \setminus U_2)=\emptyset$. By Definition~\ref{def_shield}, we know that there is some $\ve>0$ such that 
$D_0 > \frac{4}{\kappa D_1}-\frac{2}{\eta} + 2\ve$. Recall from the proof of Theorem~\ref{approximation}, that for sufficiently small $\delta>0$, $U_0\supset U_1 \supset U_2 \supset \mathcal E$ will be a scalar curvature shield for $(M, \tilde g_\delta)$, where $\tilde g_\delta$ is the metric constructed in Theorem~\ref{approximation}. But more precisely, the argument given there shows that for small enough $\delta>0$, we have the inequality
\begin{equation}\label{uniform-shield}
    \tilde D_0 > \frac{4}{\tilde \kappa \tilde D_1}-\frac{2}{\tilde \eta} + \ve,
\end{equation}
where 
$\tilde D_0$, $\tilde D_1$, $\tilde\kappa$, and $\tilde\eta^{-1}$ are uniform lower bounds, independent of $\delta$, for the corresponding quantities for $\tilde g_\delta$. Since we know that $\lim_{\delta\to0} m_\mathrm{ADM}(\tilde g_\delta)= m_\mathrm{ADM}(g)$, it suffices to show that there is a positive lower bound for $m_\mathrm{ADM}(\tilde g_\delta)$ that is independent of $\delta$. 

To see this, we select a smooth cutoff function $\eta\ge0$ compactly supported in $U_1\setminus \overline U_2$. Consider a parameter $s>0$, and consider the function $w_{\delta,s}$ solving
\[ -a\Delta_{\tilde g_\delta} w_{\delta,s} + s\eta w_{\delta,s}= 0 \text{ on }U_0, \] 
with Neumann boundary condition at $\partial U_0$. By (the proof of) Lemma~\ref{maintool}, we can see that for sufficiently small $s$, independent of $\delta$, the solution $w_{\delta,s}$ exists and satisfies a uniform bound of the form 
\begin{equation}\label{uniform-s}
\sup_{U_0} |w_{\delta,s} -1| < Cs,
\end{equation}
where the constant depends on choice of $\eta$ but is independent of $\delta$. So for small enough $s$, we can define the conformal metric $\hat g_{\delta, s} = w_{\delta,s}^{\frac{4}{n-2}} \tilde g_\delta$, which is asymptotically flat and has 
\[ R_{\hat g_{\delta, s}}= ( R_{\tilde g_\delta} - s\eta)  w_{\delta,s}^{-\frac{4}{n-2}}\]
by~\eqref{conformalformula}. So thanks to the uniform bounds~\eqref{uniform-shield} and~\eqref{uniform-s}, and by the same reasoning used at the end of the proof of Theorem~\ref{approximation}, we can see that $U_0\supset U_1 \supset U_2 \supset \mathcal E$ is still a scalar curvature shield for $\hat g_{\delta, s}$ for small enough $s$ (independent of $\delta$), which we now fix. Therefore by Theorem~\ref{shielding} and~\eqref{massformula}, we have
\[
0< m_\mathrm{ADM} (\hat g_{\delta, s})\\
=  m_\mathrm{ADM}(\tilde g_\delta)-\frac{1}{2(n-1)\omega_{n-1}}\int_M s\eta  w_{\delta,s}  \,d\mu_{\tilde g_\delta}.
\]
By~\eqref{uniform-s}, we see that this provides a uniform positive lower bound for $m_\mathrm{ADM}(\tilde g_\delta)$, as desired.
\end{proof}

\section{Proofs of main theorems}
\begin{proof}[Proof of Theorem \ref{shieldingBY}]
The proof is the same as in \cite{ShiTam} except that we use Corollary \ref{pmtwithcorners} in place of \cite[Theorem 3.1]{ShiTam}, which Shi and Tam proved using spinor methods.

Assume (either case of) the hypotheses of Theorem~\ref{shieldingBY}, but for simplicity, assume that $\Sigma$ is connected. (The argument is essentially the same if there are multiple components.) In particular, we assume that there is a connected NNSC Riemannian manifold $(\Omega, g)$, either complete or having a scalar curvature shield, whose boundary $(\Sigma, \gamma)$ can be isometrically embedded as a strictly convex hypersurface in $\mathbb{R}^n$ with mean curvature $H_0$. The main construction of Shi and Tam  \cite{ShiTam} was to construct a metric $g_+$ on 
\[\mathcal E:= [0,\infty)\times \Sigma,\]
of the form $g_+= u^2 dr^2+\sigma_r$, with the following properties:
\begin{itemize}
    \item $(\mathcal E, g)$ is a scalar-flat, asymptotically flat end;
    \item the induced metric $\sigma_0$ on $\{0\}\times\Sigma$ is equal to $\gamma$, and the
     mean curvature of $\{0\}\times\Sigma\subset \mathcal E$ with respect to the normal pointing into $\mathcal E$ is identically equal to mean curvature $H$ of $\Sigma\subset\Omega$ with respect to the normal pointing out of $\Omega$; and
    \item $m_{\mathrm{ADM}}(\mathcal E, g_+)\le \frac{1}{(n-1)\omega_{n-1}}\int_\Sigma (H_0 -H)\,d\mu_\gamma=m_{\mathrm{BY}}(\Sigma, \gamma, H)$.
\end{itemize}
If we define $(M, g)$ to be the result of appropriately gluing the two Riemannian manifolds $(\Omega, g)$ and $(\mathcal E, g_+)$ together along their isometric boundary $(\Sigma, \gamma)$ so that $(M,g)$ has a corner in the sense of Definition~\ref{corner}, then the hypotheses and the items above directly imply that $(M,g)$ satisfies the hypotheses of Corollary~\ref{pmtwithcorners}. Thus $0\le m_{\mathrm{ADM}}(\mathcal E, g_+)\le m_{\mathrm{BY}}(\Sigma, \gamma, H)$. For the equality case, we simply apply Theorem~\ref{corner-equality}.
\end{proof}

To prove Theorem~\ref{miaocomplete}, we will need to recall and slightly generalize the following fact.
\begin{thm}[Chodosh--Li~\cite{ChodoshLi}]\label{ChodoshLiThm}  Let $2\le n\le 7$, and 
Let $T^n$ denote a $n$-dimensional torus and $X^n$ an arbitrary $n$-dimensional manifold. Then $T^n\csum X^n$ cannot carry a \emph{complete} metric $g$ such that $R_g\ge0$ everywhere and $R_g>0$ somewhere.
\end{thm}
Actually, a close reading of Chodosh and Li's proof also yields the following quantitative statement: Suppose that $(M^n, g)$ is a Riemannian manifold containing a subset $M_0$ diffeomorphic to $T^n$ minus an open ball, and suppose that $(M_0, g)$ has strictly positive scalar curvature. Then there exists a constant $D>0$ depending only on the geometry of $(M_0, g)$ such that the $D$-neighborhood of $M_0$ in $(M, g)$ either fails to be compact, \emph{or} fails to have nonnegative scalar curvature.

While Theorem~\ref{ChodoshLiThm} is sufficient for proving Theorem~\ref{miaocomplete} for the case of a complete fill-in, we also need a ``shielded version'' of Theorem~\ref{ChodoshLiThm} to take care of the case of a ``shielded fill-in.''

\begin{thm}\label{ChodoshLiShield} Let $2\le n\le 7$. 
Suppose that $(M^n, g)$ is a NNSC Riemannian manifold containing a subset $M_0$ diffeomorphic to $T^n$ minus an open ball, such that  $(M_0, g)$ has strictly positive scalar curvature. Then $(M,g)$ cannot contain a scalar curvature shield $U_0\supset U_1 \supset U_2 \supset M_0$ with $\overline {U_0\setminus M_0}$ compact.
\end{thm}
\begin{rk}
In Theorem~\ref{ChodoshLiThm}, the torus $T^n$ can be replaced by the product of $S^1$ with any Schoen--Yau--Schick manifold. The same is true for Theorem~\ref{ChodoshLiShield}.
\end{rk}

\begin{proof}
The proof follows the basic strategy of Theorem~\ref{ChodoshLiThm} from~\cite{ChodoshLi}, with some modification by ideas from~\cite{LLU22, LUY21}.

Fix $\epsilon>0$. By assumption, $M_0$ is diffeomorphic to 
\[  \left\{\textbf{x}\in \mathbb{R}^n\,:\, |\textbf{x}-\textbf{k}|\ge \epsilon,\, \textbf{k}\in \mathbb{Z}^n \right\} /\, \mathbb{Z}^n,\]
where we mod out by the action of $\mathbb{Z}^n$ by addition.
So $M_0$ admits a $\mathbb{Z}$-cover $\hat M_0$ that is diffeomorphic to 
\begin{equation}
 \left\{\textbf{x}\in \mathbb{R}^n\,:\,|\textbf{x}-\textbf{k}|\ge\epsilon,\, \textbf{k}\in \mathbb{Z}^n\right\} /\, \mathbb{Z}^{n-1},
\end{equation} 
where we mod out by the action of $\mathbb{Z}^{n-1}$ by addition in the first $n-1$ components. That is, 
 $(x_1,\ldots,x_n)\sim (x_1+k_1,\ldots,x_{n-1}+k_{n-1},x_n)$ for any $(k_1,\ldots,k_{n-1}) \in\mathbb{Z}^{n-1}$. Define $\rho$ to be a $\mathbb{Z}^{n-1}$-invariant smooth function on $\mathbb R^n$ such that
 \[
 \rho(\mathbf{x})= \left\{
 \begin{array}{ll}
     x_n&\text{ if } \mathbf{x}-\mathbf{k}>2\epsilon\text{ for all }\mathbf{k}\in \mathbb{Z}^n \\
     k_n + \tfrac{1}{2}&\text{ if } \mathbf{x}-\mathbf{k}<\tfrac{3}{2}
     \epsilon\text{ for some }(k_1,\ldots,k_n)\in \mathbb{Z}^n.
 \end{array} \right.
 \]
Via diffeomorphism (and light abuse of notation), $\rho$ descends to a well-defined function on $\hat M_0$, and since $\rho$ can obviously be constructed equivariantly with the respect to the $\mathbb{Z}$ action in the last component, it is clear that there exists $L$ such that
\[ |\nabla \rho | < L,\]
 on $(\hat M_0, g)$.
 Observe that $M$ admits a   $\mathbb{Z}$-cover 
\[ \hat{M}:=\hat{M}_0\cup \left(\bigcup_{k\in \mathbb{Z}} X_k\right),\]
where each $X_k$ can be regarded as a copy of $M\setminus (\Int M_0)$, and the boundary $\partial X_k$ is identified with the component of $\partial\hat M_0$ corresponding to the $\epsilon$-sphere around $(0,...,0,k)$. Note that $\rho = k+\tfrac{1}{2}$ on $\partial X_k$. The metric $g$ lifts to a metric $\hat{g}$ on $\hat M$. 

So far this setup is the same as in~\cite{ChodoshLi} with only superficial differences. From here, \cite{ChodoshLi} constructed a function $h$ on $\hat M$ to be used as a weight for the existence of a $\mu$-bubble. For the $\mu$-bubble stability to be useful, it is important that $h$ satisfies
\begin{equation}\label{mu-condition}
    R_{\hat g} + h^2 - 2|\nabla h| >0
\end{equation}
everywhere. We will define $h$ on $\hat M_0$ as in~\cite{ChodoshLi}
but alter the definition in $X_k$. Let $\kappa>0$ be the uniform positive lower bound on $R_{\hat g}$ on $\hat{M}_0$. We define $h$ on $\hat{M}_0$ to be
\[ h(p) = -\sqrt{\kappa}\tan \left(\frac{\sqrt{\kappa}}{2L} \rho(p)\right),\]
 whenever $|\frac{\sqrt{\kappa}}{2L} \rho(p)|<\pi/2$. For values  outside that range, we define $h$ to be $\pm\infty$ in such a  way that makes $h$ a continuous function to $[-\infty, \infty]$.
A simple computation then shows that $h$ satisfies~\eqref{mu-condition} on $\hat M_0$.

Let
\[ h_k := -\sqrt{\kappa}\tan \left(\frac{\sqrt{\kappa}}{2L} (k+\tfrac{1}{2})\right)\]
denote the value of $h$ on $\partial X_k$. We need to define $h$ on $X_k$. By assumption, $X_k$ has a scalar curvature shield $U_0\supset U_1\supset U_2\supset \partial X_k$. (We will forgo putting $k$ subscripts on $U_0$, $U_1$, and $U_2$.) Let $\rho_k(p)=\dist(p, U_1)$ in $X_k$ and choose $\alpha$ to be a regular value of $\rho_k$ that is slightly less than $\frac{4\kappa}{D_1}$, where $\kappa$ is the assumed positive lower bound on $R_g$ in the region $U_1\setminus U_2$, and $D_1=\dist(U_2, \partial U_1)$. Note that $|\nabla \rho_k|\le 1$.

We consider different cases for $h_k$. When $h_k=\pm\infty$, we just define $h=\pm\infty$ on all of $X_k$. For $h_k$ finite, we consider two cases. 

\underline{Case 1:} $|h_k|< \frac{2}{\alpha}$.  We define
\begin{equation*}
h(p) = \left\{
\begin{array}{ll}
\frac{\pm 2}{\alpha - \rho_k(p)}&\text{ on }U_0\setminus U_1 \\
h_k &\text{ on }U_2 
\end{array}\right.
\end{equation*}
where  $\pm$ is the sign of $h_k$,
and then interpolate in the region $U_1\setminus U_2$. As explained in~\cite{LLU22}, this $h$ will satisfy~\eqref{mu-condition}. This is easily verified on $U_0\setminus U_1$ and $U_2$. For the interpolation region $U_1\setminus U_2$, the point is that $|\nabla h|$ can be bounded by $\frac{1}{D_1}\cdot \frac{2}{\alpha}$, and then~\eqref{mu-condition} will hold on $U_1\setminus U_2$ since $R_g\ge \kappa > \frac{D_1 \alpha}{4}$ there. 

\underline{Case 2:} $|h_k| \ge \frac{2}{\alpha}$. This case is actually simpler, because we can define $h(p)=h_k$ on all of $U_1$ and then on $U_0\setminus U_1$, we define
\[ h(p) = 
\frac{\pm 2}{\alpha - \rho_k(p)} + \left(h_k - \frac{\pm 2}{\alpha}\right)
\]
where  $\pm$ is the sign of $h_k$. Note that $h$ is continuous at $\partial U_1$ and that 
the assumption on $h_k$ guarantees $h^2 \ge \frac{4}{(\alpha-\rho_k(p))^2}$, so that 
\[ h^2 - 2|\nabla h|\ge0,\]
and hence~\eqref{mu-condition} holds.

Once we have constructed $h$ with the desired properties, the rest of the argument that leads to a contradiction is nearly the same  as in~\cite{ChodoshLi}. In the case where $\alpha<D_0 = \dist(U_1, \partial U_0)$, we can see that $h_k$ will reach $\pm\infty$ exactly as in~\cite{ChodoshLi}. Otherwise, the shielding assumption~\eqref{shield_condition} tells us that $D_0>\alpha-\frac{2}{\eta}$, where $\eta$ is the upper bound on the mean curvature of $\partial U_0$. In this case, for $h_k>0$, we see that the $h$-mean curvature of $\partial U_0$ is 
\[ H_{\partial U_0} - h \le \eta - \frac{2}{\alpha-D_0} = \left(\alpha-D_0 -\frac{2}{\eta}\right)\cdot\frac{\eta}{\alpha-D_0}<0,\]
so we can still use $\partial U_0$ as a barrier for the $\mu$-bubble problem (instead of using $h=\infty$ there), and the proof proceeds in the same manner. (For $h_k<0$, the argument is the same but we reverse the normal on $\partial U_0$.)

In summary, there must exist a stable $\mu$-bubble with respect to the weight $h$. The condition~\eqref{mu-condition} guarantees that the induced metric on this stable $\mu$-bubble is conformal to one with positive scalar curvature.\footnote{Because of where the barriers are, the $\mu$-bubble has to intersect $\hat M_0$, where the inequality~\eqref{mu-condition} is strict. This is why we can obtain positive scalar curvature rather than just nonnegative scalar curvature.} But we have enough control over the topology of the $\mu$-bubble (as in the standard Schoen--Yau descent argument~\cite{SY79}) to obtain a contradiction. See~\cite{ChodoshLi} for details.
 
\end{proof}


Finally, we recall a crucial theorem from \cite{shi3} which will be used to prove Theorem~\ref{miaocomplete}.

\begin{thm}[Theorem 1.1 of~\cite{shi3}]\label{cobordismlemma}
Let $X$ be a compact manifold with nonempty boundary $Y = \partial X$. Then any metric $\gamma$ on $Y$ can be extended to a Riemannian
metric $g$ on $X$ with positive scalar curvature.
\end{thm}

\begin{proof}[Proof of Theorem~\ref{miaocomplete}]
We adapt the idea of \cite{Miao20} to our setting. Let $(\Sigma^{n-1}, \gamma)$ be our given closed Riemannian manifold with $n\le 7$.
Let $X^n$ be the smooth manifold obtained by taking $\Sigma^{n-1}\times [0,1]$ and attaching a torus $T^n$ via a connected sum at an interior point of the product. By Theorem~\ref{cobordismlemma}, there exists a metric $h$ on $X$ such that $R_h>0$ everywhere and $h$ induces the given metric $\gamma$ on both boundary components of $X$.
Let $H$ denote the mean curvature of $\partial X$ with respect to the metric $h$ and the normal pointing into $X$. We set 
\[\lambda:=\max_{\partial X}H.\]

We claim that this  is the desired $\lambda$ in the conclusion of Theorem~\ref{miaocomplete}, which we will prove by contradiction. Suppose $(\Omega,g)$ is a NNSC fill-in of $(\Sigma,\gamma,\eta)$ satisfying the hypotheses of the theorem, but for which 
\begin{equation}\eta >\lambda\text{ on }\Sigma.\label{MCineq}\end{equation}
Here $\eta$ denotes the mean curvature of $\partial\Omega$ with respect to the outward pointing normal. Let $(\Omega_1,g)$ and $(\Omega_2, g)$ denote two copies of $(\Omega, g)$ and glue $\partial\Omega_1$ to one component  of $\partial X$ and  $\partial\Omega_2$ to the other component of $\partial X$. This results in a 
 new Riemannian manifold $\Omega_1\cup X\cup \Omega_2$ whose glued metric has a corner along $\partial X$ in the sense of Definition \ref{corner}, and the condition \eqref{MCineq} implies the metric satisfies the ``distributional NNSC" hypothesis (appearing in Theorem~\ref{approximation}) across the corner $\partial X$. 

To obtain the contradiction without any more work, we can argue as follows: By construction, there is point $p$ in the interior of $X$ with strictly positive scalar curvature. We now take a complete asymptotically flat Riemannian manifold $N$ with positive scalar curvature\footnote{For a concrete choice, one can take a $t=0$ slice of a Reissner--Nordstr\"om black hole with $0<|e|\le m$.} and perform a Gromov--Lawson PSC connected sum joining $N$ to $\Omega_1\cup X\cup \Omega_2$ by deleting a small neighborhood of $p$.  The new Riemannian manifold $Y=(\Omega_1\cup X\cup\Omega_2)\csum N$ constructed in this way satisfies the hypotheses of our NNSC approximation theorem (Theorem \ref{approximation}), yielding a new smooth metric $g'$ on $Y$.
In the shielded case, we can approximate and apply our Theorem \ref{ChodoshLiShield} to $(Y, g')$ to obtain a contradiction. In the complete case, we obtain from the approximation theorem a complete Riemannian manifold $(Y, g')$ with scalar curvature that is nonnegative everywhere and positive on a nonempty open set. By Theorem \ref{ChodoshLiThm}, we have reached a contradiction. 

The step of attaching an asymptotically flat end was only done for convenience -- the corner along $\partial X$ can be smoothed out directly. We refer the reader to the appendix for this argument.
\end{proof}

\begin{appendix}

\section{Miao smoothing for complete manifolds}

In this appendix, we show how our adaptation of Miao's ideas in Section \ref{pmtsection} can be applied to complete manifolds with NNSC along a compact corner, without any assumption of asymptotic flatness. Another approach might be to work by hand and use the local $h$-principle, see \cite{G18, GromovFour}.

We will require the following result of Kazdan \cite{Kazdan}. 

\begin{prop}\label{Kazdan}
Let $(M^n,g)$ be a complete Riemannian manifold, and let $V\in C^\infty(M)$. Suppose there exists a bounded domain $\Omega\subset M$ such that $\spt V^-\subset\Omega$ and $\mu_1(\Omega)>0$, where $\mu_1$ denotes the principal Neumann eigenvalue of $-\Delta+V$ on $\Omega$. Then there exists a function $u$ on $M$ and a constant $\delta>0$ such that $\delta<u<\delta^{-1}$ and $-\Delta u+Vu>0$.

In particular, if the scalar curvature of $(M,g)$ is nonnegative outside of a bounded domain $\Omega$ and the principal Neumann eigenvalue of the conformal Laplacian on $\Omega$ is positive, then $g$ is conformal to a complete metric with positive scalar curvature (PSC). 
\end{prop}
Using this fact, we can prove the following:
\begin{prop}
Let $(M^n,g)$ be a complete manifold admitting a corner along the closed hypersurface $\Sigma^{n-1}$ in the sense of Definition \ref{corner}. Suppose $R_g\ge 0$ on $M\setminus \Sigma$ and $H_+\le H_-$. Then either 
\begin{enumerate}
    \item $\Ric_g\equiv 0$ on $M\setminus\Sigma$ and $H_+\equiv H_-$ on $\Sigma$, or
    \item $M$ admits a complete PSC metric, with respect to some $C^{3,\alpha}$-differentiable structure.
\end{enumerate}
\end{prop}
\begin{proof}
 We consider three cases: (a) $R_g>0$ somewhere away from $\Sigma$, (b) $R_g\equiv 0$ but $H_+<H_-$ somewhere on $\Sigma$, and (c) $R_g\equiv 0$ and $H_+\equiv H_-$ but $|\Ric_g|>0$ somewhere in $M\setminus\Sigma$. These exhaust all possible situations when (1) is false, so in all three cases we will show that we are in case (2). 

We begin the proof of cases (a) and (b) in essentially the same way. In either case, $\Omega$ will be a connected, bounded domain containing $\Sigma$. In case (a), we will additionally assume that $\Omega$ contains the closure of an open ball $B$ on which we have $R_g\ge\kappa>0$, for some constant $\kappa$. 

Let $g_\delta$ denote the Riemannian metric obtained from applying Miao's smoothing Lemma \ref{miaosmoothening} to the current setting.
With this $\Omega$ fixed, let $\mu_\delta$ denote the principal Neumann eigenvalue of the conformal Laplacian of $g_\delta$ on $\Omega$. If $\mu_\delta>0$ for some $\delta>0$, then $M$ admits a complete PSC metric by Proposition~\ref{Kazdan}. In fact, we will prove the stronger statement that $\mu_\delta>0$ for all sufficiently small $\delta>0$. To the contrary, suppose we can find a sequence  $\delta_i\to0$ such that $\mu_{\delta_i}\le 0$. Let $u_i$ be the positive $L^2$-normalized Neumann eigenfunction associated to $\mu_{\delta_i}$, and let $R_i$ denote the scalar curvature of $g_{\delta_i}$. Integrating the eigenvalue equation and using the Neumann condition, we have
\begin{equation}\label{eigenvalue}
\int_\Omega a|\nabla u_i|^2+\int_\Omega R_i u_i^2 = \mu_{\delta_i}.
\end{equation}
Since  we assumed $\mu_{\delta_i}\le0$, it follows that
\[\int_\Omega a|\nabla u_i|^2+\int_\Omega R^+_{i}u_i^2\le \int_\Omega R^-_{i}u_i^2.\]
Applying the Sobolev inequality to the left-hand side and H\"older's inequality to the right-hand side, we have
\[\left(\int_\Omega |u_i-\overline{u_i}|^\frac{2n}{n-2}\right)^{\frac{n-2}{n}}+\int_\Omega R_{i}^+u_i^2 \le C\left(\int_\Omega |R_{i}^{-}|^\frac{n}{2}\right)^\frac{2}{n}\left(\int_\Omega |u_i|^\frac{2n}{n-2}\right)^{\frac{n-2}{n}}.\]
By Minkowski's inequality, we have 
\[\left(\int_\Omega |u_i|^\frac{2n}{n-2}\right)^{\frac{n-2}{n}}\le C\left(|\overline{u_i}|+\left(\int_\Omega |u_i-\overline{u_i}|^\frac{2n}{n-2}\right)^{\frac{n-2}{n}}\right) .\]
By H\"older's inequality (and compactness of $\Omega$), we have
\[|\overline{u_i}|\le C \|u_i\|_{L^1(\Omega)}\le C\|u_i\|_{L^2(\Omega)}=C.\]
Since $\delta_i\to0$, Lemma~\ref{miaosmoothening} implies that
\[\left(\int_\Omega |R_{i}^{-}|^\frac{n}{2}\right)^\frac{2}{n}\to 0.\]
Combined with the three previous inequalities, it is straightforward to conclude that
\begin{equation}\int_\Omega |u_i-\overline{u_i}|^\frac{2n}{n-2}\to 0\quad\text{and}\quad \int_\Omega R_{i}^+u_i^2\to 0.\label{A1}\end{equation}
Since $\overline{u_i}$ is bounded, we may assume  it converges to a real number $A$ after passing to a subsequence.   By the triangle inequality, $u_i\to A$ in $L^\frac{2n}{n-2}$, and consequently, also in $L^2$ by H\"older's inequality, since $\frac{2n}{n-2}\ge 2$. It follows that $\|A\|_{L^2(\Omega)}=1$, so $A\ne 0$.

To obtain a contradiction to \eqref{A1} in case (a), we simply note that
\[\int_\Omega R_i^+u_i^2\ge \kappa\int_B u_i^2\to \kappa A^2 \vol(B)>0,\]
which contradicts the second statement in \eqref{A1}. 

In case (b), we contradict \eqref{A1} using an argument of Miao \cite{Miao02}. We use the strict mean curvature gap to force a concentration of scalar curvature as follows. First, we note that $u_i$ is harmonic outside of $U_{\delta_i}$, so by local elliptic theory $u_i\to A$ in $C^{2,\alpha}_\loc(\Omega\setminus\Sigma)$. We may now appeal to the argument of equations (65)-(76) in \cite{Miao02}, where we note that his argument does not require ``$v=1$", but in fact works for any nonzero constant, such as $A$ in our case. Miao's calculations show that 
\[\int_\Omega R_i^+ u_i^2\]
is strictly bounded below, which is a contradiction. 

Finally, in case (c), let $B$ be an open ball such that $\overline B$  does not intersect $\Sigma$ and 
 $|\Ric_g|>0$ on $\overline B$. Let $\chi$ be a nontrivial nonnegative bump function with support in $B$. Let $\overline g_t=g_{t^2}+t\chi\Ric_g$, where $g_{t^2}$ is the Miao smoothing with parameter $\delta=t^2$, just as we did in the proof of Theorem~\ref{corner-equality}. Let $\mu_t$ be the principal Neumann eigenvalue of the conformal Laplacian of $\overline g_t$ on $\Omega$, where $\Omega$ is a bounded domain containing $\Sigma$ and $\overline B$. As in cases (a) and (b), we claim that for sufficiently small $t>0$, $\mu_t>0$. Again, this would imply that $M$ carries a complete PSC metric.
 To the contrary, suppose that there exists some sequence of times $t_i\to0$ such that 
  $\mu_{t_i}\le 0$. Let $u_i$ be the positive $L^2$-normalized Neumann eigenfunction associated to $\mu_{t_i}$, and let $R_i$ denote the scalar curvature of~$\overline g_{t_i}$.

Repeating the arguments of cases (a) and (b), we find a subsequence 
and a number $A\ne 0$ such that $u_i\to A$ in $L^2$. 
Let $K$ be a compact subset of $\Omega$ contaning $\Sigma$ and $\overline B$.
If we can show that $u_i\to A$ in $L^\infty(K)$, then just as in the proof of Theorem~\ref{corner-equality},
\[\int_\Omega\frac{R_{i}}{t_i}u_i^2\to A^2\int_B\chi |\Ric_g|^2>0,\]
which would contradict the inequality $\mu_{t_i}\le 0$ because of equation~\eqref{eigenvalue}, but with $t_i$ in place of $\delta_i$.

It only remainds to prove that $u_i\to A$ in $L^\infty(K)$. Note that $|R_i|\le C$ by virtue of the condition $H_+\equiv H_-$ and parts (5) and (6) of Lemma~\ref{miaosmoothening}. It follows that 
\[\int_\Omega R_{i} v^2 \ge -C\int_\Omega v^2\]
for any $v\in L^2(\Omega)$, whence $\mu_{t_i}\ge -C$. Therefore, the functions $u_i$ satisfy the equation 
\[-a\Delta_{\overline g_{t_i}}u_i+(R_{i}-\mu_{t_i})u_i=0\] and $\sup_\Omega|R_{i}-\mu_{t_i}|\le C$. Then the De Giorgi--Nash--Moser estimate \cite[Theorem 8.17]{GT} implies that for some constant $C$,
\[\sup_K u_i\le C.\] 
By~\eqref{eigenvalue},
\[0\ge \mu_{t_i}\ge -\int_K R_{i}^-u_i^2 \ge -C\int_K R_{i}^-.\]
Since we know the right-hand side approaches zero by Lemma~\ref{miaosmoothening}, 
we see that $\mu_{t_i}\to 0$. The quantity $u_{i}-A$ solves the equation
\begin{equation}\label{eqn:DGN}
-a\Delta_{\overline g_{t_i}}(u_{i}-A)+(R_{i}-\mu_{t_i})(u_{i}-A)=-A(R_{i}-\mu_{t_i}).
\end{equation}
By the case (c) assumption, Lemma~\ref{miaosmoothening} tells us that $\|R_i\|_{L^n(\Omega)}\to 0$. Combined with $\mu_{t_i}\to 0$, we see that the right-side of equation~\eqref{eqn:DGN} converges to zero in $L^n(\Omega)$, so by another application of the De Giorgi--Nash--Moser estimate, we obtain
\[\sup_K |u_{i}-A|\to 0,\]
as claimed.
\end{proof}

\end{appendix}

\bibliographystyle{alpha}
\bibliography{noncompact-fill-ins}

\begin{thebibliography}{SWWZ21}

\bibitem[AC96]{AnderssonChrusciel}
Lars Andersson and Piotr~T. Chru\'{s}ciel.
\newblock Solutions of the constraint equations in general relativity
  satisfying ``hyperboloidal boundary conditions''.
\newblock {\em Dissertationes Math. (Rozprawy Mat.)}, 355:100, 1996.

\bibitem[BC03]{BartnikChrusciel}
Robert Bartnik and Piotr~T. Chru\'{s}ciel.
\newblock Boundary value problems for {D}irac--type equations, with
  applications.
\newblock {\em \emph{arXiv:math/0307278}}, 2003.

\bibitem[CL20]{ChodoshLi}
Otis Chodosh and Chao Li.
\newblock Generalized soap bubbles and the topology of manifolds with positive
  scalar curvature.
\newblock {\em \emph{arXiv:2008.11888}}, 2020.

\bibitem[CLSZ21]{CLSZ21}
Jie Chen, Peng Liu, Yuguang Shi, and Jintian Zhu.
\newblock Incompressible hypersurface, positive scalar curvature and positive
  mass theorem.
\newblock {\em \emph{arXiv:2112.14442}}, 2021.

\bibitem[CLZ22]{CLZ22}
Jianchun Chu, Man-Chun Lee, and Jintian Zhu.
\newblock Singular positive mass theorems with arbitrary ends.
\newblock {\em \emph{arXiv:2210.08261}}, 2022.

\bibitem[CZ21]{CecchiniZeidler21}
Simone Cecchini and Rudolf Zeidler.
\newblock The positive mass theorem and distance estimates in the spin setting.
\newblock {\em \emph{arXiv:2108.11972}}, 2021.

\bibitem[DK81]{DTK}
Dennis~M. DeTurck and Jerry~L. Kazdan.
\newblock Some regularity theorems in {R}iemannian geometry.
\newblock {\em Ann. Sci. \'{E}cole Norm. Sup. (4)}, 14(3):249--260, 1981.

\bibitem[EMW12]{EichmairMiaoWang}
Michael Eichmair, Pengzi Miao, and Xiaodong Wang.
\newblock Extension of a theorem of shi and tam.
\newblock {\em Calc. Var. PDE}, 43:45--56, 2012.

\bibitem[Gro18]{G18}
Misha Gromov.
\newblock Metric inequalities with scalar curvature.
\newblock {\em GAFA}, 28(6):645--726, 2018.

\bibitem[Gro21]{GromovFour}
Misha Gromov.
\newblock Four lectures on scalar curvature.
\newblock {\em \emph{arXiv:1908.10612v4}}, 2021.

\bibitem[GT83]{GT}
David Gilbarg and Neil~S. Trudinger.
\newblock {\em Elliptic partial differential equations of second order}, volume
  224 of {\em Grundlehren der Mathematischen Wissenschaften [Fundamental
  Principles of Mathematical Sciences]}.
\newblock Springer-Verlag, Berlin, second edition, 1983.

\bibitem[Her98]{Herzlich}
Marc Herzlich.
\newblock The positive mass theorem for black holes revisited.
\newblock {\em J. Geom. Phys.}, 26(2), 1998.

\bibitem[Kaz82]{Kazdan}
Jerry~L. Kazdan.
\newblock Deformation to positive scalar curvature on complete manifolds.
\newblock {\em Math. Ann.}, 261(2):227--234, 1982.

\bibitem[Lee19]{Lee}
Dan~A. Lee.
\newblock {\em Geometric Relativity}.
\newblock \emph{Graduate Studies in Mathematics, Vol. 201, AMS}, 2019.

\bibitem[LLU22a]{LLU22}
Dan~A. Lee, Martin Lesourd, and Ryan Unger.
\newblock Density and positive mass theorems for incomplete manifolds.
\newblock {\em \emph{arXiv:2201.01328}}, 2022.

\bibitem[LLU22b]{LLU21}
Dan~A. Lee, Martin Lesourd, and Ryan Unger.
\newblock Density and positive mass theorems for initial data sets with
  boundary.
\newblock {\em \emph{To appear in} Comm. Math. Phys.}, 2022.

\bibitem[LUY21]{LUY21}
Martin Lesourd, Ryan Unger, and Shing-Tung Yau.
\newblock The positive mass theorem with arbitrary ends.
\newblock {\em \emph{To appear in} J. Differential Geom.}, 2021.

\bibitem[Mia02]{Miao02}
Pengzi Miao.
\newblock Positive mass theorem on manifolds admitting corners along a
  hypersurface.
\newblock {\em Adv. Theor. Math. Phys.}, 6:1163–1182, 2002.

\bibitem[Mia21]{Miao20}
Pengzi Miao.
\newblock Nonexistence of nnsc fill-ins with large mean curvature.
\newblock {\em Proc. Amer. Math. Soc.}, 149:2705--2709, 2021.

\bibitem[MM17]{MantoulidisMiao}
Christos Mantoulidis and Pengzi Miao.
\newblock Total mean curvature, scalar curvature, and a variational analog of
  brown–york mass.
\newblock {\em Commun. Math. Phys.}, 352:703–718, 2017.

\bibitem[MS12]{McFeron-Szekylihidi}
Donovan McFeron and G\'{a}bor Sz\'{e}kylihidi.
\newblock On the positive mass theorem for manifolds with corners.
\newblock {\em Comm. Math. Phys.}, 313:425–443, 2012.

\bibitem[Nir53]{Nirenberg}
Louis Nirenberg.
\newblock The {W}eyl and {M}inkowski problems in differential geometry in the
  large.
\newblock {\em Comm. Pure Appl. Math.}, 6:337–394, 1953.

\bibitem[Pog52]{Pogorelov}
Aleksei~Vasil'evich Pogorelov.
\newblock Regularity of a convex surface with given gaussian curvature.
\newblock {\em (Russian) Mat. Sbornik N.S.}, 31(73):88--103, 1952.

\bibitem[ST02]{ShiTam}
Yuguang Shi and Luen-Fai Tam.
\newblock Positive mass theorem and the boundary behaviors of compact manifolds
  with nonnegative scalar curvature.
\newblock {\em J. Differential Geom.}, 62(1):79--125, 2002.

\bibitem[SWW22]{shi3}
Yuguang Shi, Wenlong Wang, and Guodong Wei.
\newblock Total mean curvature of the boundary and nonnegative scalar curvature
  fill-ins.
\newblock {\em J. Reine Angew. Math.}, 784:215--250, 2022.

\bibitem[SWWZ21]{shi1}
Yuguang Shi, Wenlong Wang, Guodong Wei, and Jintian Zhu.
\newblock On the fill-in of nonnegative scalar curvature metrics.
\newblock {\em Math. Ann.}, 379(1-2):235--270, 2021.

\bibitem[SY79a]{SY79PMT}
Richard Schoen and Shing-Tung Yau.
\newblock On the proof of the positive mass conjecture in general relativity.
\newblock {\em Comm. Math. Phys.}, 65(1):45--76, 1979.

\bibitem[SY79b]{SY79}
Richard Schoen and Shing-Tung Yau.
\newblock On the structure of manifolds with positive scalar curvature.
\newblock {\em Manuscripta Math.}, 28(1-3):159--183, 1979.

\bibitem[SY17]{SY17}
R.~Schoen and S-T. Yau.
\newblock Positive scalar curvature and minimal hypersurface singularities.
\newblock {\em \emph{arXiv:1704.05490}}, 2017.

\bibitem[Wit81]{Witten}
Edward Witten.
\newblock A new proof of the positive energy theorem.
\newblock {\em Comm. Math. Phys.}, 80(3):381--402, 1981.

\bibitem[Zhu22]{Zhu22}
Jintian Zhu.
\newblock Positive mass theorem with arbitrary ends and its application.
\newblock {\em \emph{To appear in} Int. Math. Res. Not.}, 2022.

\end{thebibliography}

\end{document}